\documentclass[11pt, reqno, oneside]{amsart}

\usepackage{amsmath,amssymb,amsthm}
\usepackage{txfonts, mathrsfs}
\usepackage[utf8]{inputenc}
\usepackage{a4}
\usepackage[english]{babel}
\usepackage{enumerate}
\usepackage{graphicx}
\usepackage{bbm}
\usepackage[hidelinks]{hyperref}
\usepackage{enumitem}
\usepackage{tikz-cd}

\numberwithin{equation}{section}

\newtheorem{theorem}{Theorem}[section]
\newtheorem{proposition}[theorem]{Proposition}
\newtheorem{lemma}[theorem]{Lemma}

\newtheorem{corollary}[theorem]{Corollary}

\theoremstyle{definition}

\newtheorem{definition}[theorem]{Definition}
\newtheoremstyle{customNumber}
     {}          
     {}          
     {\itshape}  
     {}          
     {\bfseries} 
     {.}         
     { }         
     {\thmname{#1}\thmnumber{ #2}\thmnote{ #3}}
\theoremstyle{customNumber}

\renewcommand{\phi}{\varphi}
\renewcommand{\rho}{\varrho}
\newcommand{\norm}[1]{\lVert#1\rVert}
\newcommand{\abs}[1]{\lvert#1\rvert}
\newcommand{\diam}{\operatorname{diam}}

\newcommand{\on}{\:\mbox{\rule{0.1ex}{1.2ex}\rule{1.1ex}{0.1ex}}\:}
\newcommand{\bb}[1]{\llbracket #1\rrbracket}

\DeclareMathOperator{\N}{\mathbb{N}}
\DeclareMathOperator{\R}{\mathbb{R}}
\DeclareMathOperator{\Z}{\mathbb{Z}}

\DeclareMathOperator{\Lip}{Lip}

\DeclareMathOperator{\sgn}{sgn}

\DeclareMathOperator{\spt}{spt}
\DeclareMathOperator{\st}{st}
\DeclareMathOperator{\interior}{int}
\DeclareMathOperator{\mass}{\mathbf{M}}
\DeclareMathOperator{\bI}{\mathbf{I}}
\DeclareMathOperator{\bM}{\mathbf{M}}
\DeclareMathOperator{\bN}{\mathbf{N}}
\DeclareMathOperator{\id}{id}

\DeclareMathOperator{\Hull}{Hull}

\DeclareMathOperator{\vol}{Vol}

\DeclareMathOperator{\Haus}{\mathscr{H}}
\DeclareMathOperator{\IC}{IC}
\DeclareMathOperator{\set}{set}

\DeclareMathOperator{\apmd}{apmd}
\DeclareMathOperator{\md}{md}
\DeclareMathOperator{\ap}{{\rm ap}}

\newdimen\vintkern\vintkern11pt
\def\vint{-\kern-\vintkern\int}

\renewcommand{\epsilon}{\varepsilon}
\def\Xint#1{\mathchoice
{\XXint\displaystyle\textstyle{#1}}%
{\XXint\textstyle\scriptstyle{#1}}%
{\XXint\scriptstyle\scriptscriptstyle{#1}}%
{\XXint\scriptscriptstyle\scriptscriptstyle{#1}}%
\!\int}
\def\XXint#1#2#3{{\setbox0=\hbox{$#1{#2#3}{\int}$ }
\vcenter{\hbox{$#2#3$ }}\kern-.6\wd0}}

\def\dashint{\Xint-}

\usepackage{etoolbox}

\makeatletter
\patchcmd{\@setaddresses}{\indent}{\noindent}{}{}
\patchcmd{\@setaddresses}{\indent}{\noindent}{}{}
\patchcmd{\@setaddresses}{\indent}{\noindent}{}{}
\patchcmd{\@setaddresses}{\indent}{\noindent}{}{}
\makeatother

\title[Geometric and analytic structures on metric manifolds]{Geometric and analytic structures on metric spaces homeomorphic to a manifold}

\thanks{This research was partially supported by Swiss National Science Foundation grant 212867. D.~M.~was moreover supported by ERC Starting Grant 713998 GeoMeG}

\author{Giuliano Basso}

\address{Max Planck Institute for Mathematics \\ Vivatsgasse 7 \\ 53111 Bonn \\ Germany}
\email{basso@mpim-bonn.mpg.de}

\author{Denis Marti}

\address{Department of Mathematics\\ University of Fribourg\\  Chemin du Mus\'ee 23\\  1700 Fribourg, Switzerland}
\email{denis.marti@unifr.ch}

\author{Stefan Wenger}

\address{Department of Mathematics\\ University of Fribourg\\  Chemin du Mus\'ee 23\\   1700 Fribourg, Switzerland}
\email{stefan.wenger@unifr.ch}

\begin{document}

\begin{abstract}
We study metric spaces homeomorphic to a closed oriented manifold from both geometric and analytic perspectives. We show that such spaces (which are sometimes called metric manifolds) admit a non-trivial integral current without boundary, provided they satisfy some weak assumptions. The existence of such an object should be thought of as an analytic analog of the fundamental class of the space and can also be interpreted as giving a way to make sense of Stokes' theorem in this setting. Using our existence result, we establish that Riemannian manifolds are Lipschitz-volume rigid among certain metric manifolds and we show the validity of (relative) isoperimetric inequalities in metric $n$-manifolds that are Ahlfors $n$-regular and linearly locally contractible. The former statement is a generalization of a well-known Lipschitz-volume rigidity result in Riemannian geometry and the latter yields a relatively short and conceptually simple proof of a deep theorem of Semmes about the validity of Poincar\'e inequalities in these spaces. Finally, as a further application, we also give sufficient conditions for a metric manifold to be rectifiable.
\end{abstract}

\maketitle
\vspace{-0.2em}

\section{Introduction}

\subsection{Overview} 
In this article we study metric spaces that are homeomorphic to closed oriented manifolds. We are interested in the question whether such spaces admit a non-trivial metric integral current without boundary and what are implications of the existence of such a current. In Euclidean space, integral currents were first introduced and studied by Federer--Fleming around 1960 in their seminal paper \cite{federer-1960} in connection with Plateau's problem of finding area minimizing surfaces of any dimension with prescribed boundary. Nowadays, they have become a standard tool in geometric measure theory, with a wide range of applications that go well beyond area minimization problems. We will work with an extension of the theory to the setting of complete metric spaces due to Ambrosio and Kirchheim \cite{ambrosio-kirchheim-2000}. This theory provides a rich and powerful framework with a diverse range of applications, see \cite{Faessler-Orponen-2019, Huang-Kleiner-Stadler-2022, Kleiner-Lang-2020,  Song-2023, Sormani-Wenger-2011, Wenger-2008}. 

We prove two main existence results, Theorems~\ref{thm:main} and \ref{thm:existence-intcurr-dimension2-intro}. These state that any metric \(n\)-manifold satisfying weak assumptions admits an essentially unique integral \(n\)-current without boundary. For example, when \(n=2\), the only requirement we impose on the metric manifold is finite Hausdorff 
\(2\)-measure.  The so obtained current shares many properties with the fundamental class of the space. Most interestingly, the existence of this “metric fundamental class" leads to implications concerning metric manifolds not involving any currents. These results include the validity of Poincar\'e inequalities for metric manifolds and a Lipschitz-volume rigidity theorem. 

Poincar\'e inequalities are of vital importance in the field of analysis on metric spaces, and many aspects of first order calculus have been developed in the setting of doubling metric measure spaces supporting such inequalities, see for example \cite{Cheeger-diff-Lip-1999, Hajlasz-Koskela-1995, Heinonen-Koskela-qc-Acta-1998, Heinonen-Koskela-Shanmugalingam-Tyson-2015}. We use our main existence results to prove a relative isoperimetric inequality in metric \(n\)-manifolds that are Ahlfors $n$-regular and linearly locally contractible; see Theorem~\ref{thm:relative-isoperimetric-inequality}. It is well-known that such a result implies a Poincar\'e inequality. This yields a conceptually simple and short proof of a deep theorem of Semmes \cite{semmes-1996}.

Lipschitz-volume rigidity results are generalizations of the following statement regarding closed, oriented, Riemannian manifolds \(X\) and \(Y\) of the same dimension. If \(X\) and \(Y\) have the same volume, then any \(1\)-Lipschitz map of degree one from \(X\) to \(Y\) is an isometric homeomorphism, see \cite{besson-1995, burago--1995, burago-2010}. Our existence theorems, combined with Z\"ust's recent work \cite{Zuest-23}, directly yield a Lipschitz-volume rigidity theorem for metric manifolds; see Theorem~\ref{thm:Lipschitz-Volume-rigidity}. When restricted to orientable metric surfaces this recovers a recent result of \cite{Meier-Ntalampekos-23}.

\subsection{Statement of main results}
For $n\geq 0$ we denote by $\bI_n(X)$ the abelian group of $n$-dimensional metric integral currents on a complete metric space $X$ in the sense of Ambrosio--Kirchheim \cite{ambrosio-kirchheim-2000}. We refer to Section~\ref{sec:currents} for the relevant definitions. Any Lipschitz map $\varphi\colon X\to Y$ naturally induces a homomorphism $\varphi_\#\colon\bI_n(X)\to \bI_n(Y)$ for every $n\geq 0$. Moreover, every Lipschitz $n$-chain in $X$ induces an $n$-dimensional metric integral current in $X$. However, not every integral current is of this form. There is a natural boundary operator $\partial$ for metric integral currents which makes 
 \[
\dotsm \overset{\partial}{\longrightarrow} \bI_n(X) \overset{\partial}{\longrightarrow} \bI_{n-1}(X) \overset{\partial}{\longrightarrow} \dotsm \overset{\partial}{\longrightarrow} \bI_0(X) 
\]
into a chain complex. The associated $n$-th homology group is denoted by $H_n^{\IC}(X)$ and called the $n$-th homology group via integral currents. 

Let $M$ denote a closed, oriented, Riemannian $n$-manifold. Throughout this article all manifolds will be assumed connected and without boundary. Then $H_n^{\IC}(M)$ is isomorphic to the $n$-th singular homology group of $M$ and thus to $\Z$. The current $\bb{M}$ given by integration (of $n$-forms) over $M$ is a generator of $H_n^{\IC}(M)$. However, for an arbitrary metric space \(X\) the homology groups via integral currents are in general not isomorphic to the singular homology groups, except under very restrictive conditions on $X$, see \cite{mitsuishi-2019, riedweg2009singular}. 

To each integral \(n\)-current \(T\) belongs an associated finite Borel measure $\norm{T}$, called the mass measure of $T$. This measure can be thought of as the volume measure on the generalized surface $T$. As a prime example, for a Riemannian \(n\)-manifold as above, the mass measure of \(\bb{M}\) is equal to the Riemannian volume measure on \(M\). 

Our first theorem provides conditions on a metric manifold 
under which the top-dimensional homology group via integral currents is isomorphic to the singular homology group and the mass measure of the generator is proportional to the Hausdorff measure. In order to state the theorem, recall that a metric space $X$ is said to be \textit{linearly locally contractible} if there exists $\Lambda\geq 1$ such that every ball of radius $0< r< \Lambda^{-1}\diam X$ in $X$ is contractible inside the ball of radius $\Lambda r$ with the same center. This condition rules out the existence of cusps and neck pinches, and is for example preserved under quasisymmetric homeomorphisms. 

\begin{theorem}\label{thm:main}
Let $X$ be a metric space which has finite Hausdorff $n$-measure and is homeomorphic to a closed, oriented, smooth $n$-manifold $M$. If $X$ is linearly locally contractible, then the \(n\)-th homology group \(H_n^{\IC}(X)\) via integral currents is infinite cyclic and it admits a generator \(T\) such that \(\phi_{\#}T=\deg(\phi)\cdot \bb{M}\) for every Lipschitz map \(\phi\colon X \to M\). Moreover, there exist $C$, $c>0$ such that \(\norm{T} \leq C \cdot \Haus^n\) and
\begin{equation}\label{eq:lower-bound-1}
\norm{T}(B(x, r)) \geq c \cdot r^n
\end{equation}
for every \(x\in X\) and every \(0\leq r \leq \diam X\).
\end{theorem}

Here, $M$ is equipped with any Riemannian metric, $\deg(\varphi)$ is the (topological) degree of $\varphi$, and \(\Haus^n\) denotes the Hausdorff \(n\)-measure on \(X\). The constant $C$ in the theorem only depends on the dimension $n$, while $c$ depends on $n$ and the linear local contractibility constant. The assumption that \(X\) is linearly locally contractible frequently appears in geometry and analysis. In particular, in the context of uniformization problems for metric spaces, quasiconformal geometry, the study of Gromov hyperbolic groups, and finiteness theorems in geometry, see for example \cite{Bonk-Kleiner-2002, david-american-2016, Ferry-1994, Grove-Petersen-Wu-1990, Heinonen-ICM, Heinonen-Koskela-qc-Acta-1998, heinonen-sullivan-2002,  Kleiner-ICM, Laakso-Ainfty-2002, Semmes-good-metric-spaces-1996}. 

Results similar to our Theorem~\ref{thm:main} were obtained by Heinonen--Sullivan \cite{heinonen-sullivan-2002} and Kirsil\"{a} \cite{kirsila-2016} in a much more restricted setting; see also \cite{heinonen-keith-2011} for an application.
We remark that it follows directly from the bounds in the theorem that the mass measure $\norm{T}$ is proportional to $\Haus^n$. Moreover, as an easy consequence, we obtain the following rectifiability result. Recall that a metric space is called \textit{$n$-rectifiable} if it can be covered, up to an $\Haus^n$-negligible set, by countably many Lipschitz images of subsets of $\R^n$.

\begin{corollary}\label{cor:recitifiability}
 Let $X$ be a metric space which has finite Hausdorff $n$-measure and is homeomorphic to a closed, orientable, smooth $n$-manifold. If $X$ is linearly locally contractible, then $X$ is $n$-rectifiable and there exists $c>0$ such that every ball in $X$ of radius \(0 \leq r \leq \diam X\) has Hausdorff $n$-measure at least $c\cdot r^n$.
\end{corollary}

The constant \(c\) in the corollary depends only on \(n\) and the linear local contractibility constant of \(X\). The asserted lower bound on the Hausdorff measure of balls generalizes a result of Kinneberg \cite{Kinneberg-2018} who proved an analogous result with the additional assumption that the metric manifold is doubling. We remark that although this lower bound formally follows directly from Theorem~\ref{thm:main}, we actually prove it independently in Theorem~\ref{thm:lower-bound-Hausdorff-measure} along the way to the proof of Theorem~\ref{thm:main}. 

More substantial applications of Theorem~\ref{thm:main} can be found in Section~\ref{sec:applications} below. Corollary~\ref{cor:recitifiability} is false without the assumption that $X$ is linearly locally contractible. This follows from \cite[Appendix A]{Sormani-Wenger-calcvar-2010}. The same example also shows that the lower bound on mass \eqref{eq:lower-bound-1} in Theorem~\ref{thm:main} is false without the linear local contractibility assumption. Our next result guarantees existence of a non-trivial cycle in dimension \(2\) without assuming linear local contractibility.

\begin{theorem}\label{thm:existence-intcurr-dimension2-intro}
Let $X$ be a metric space of finite Hausdorff $2$-measure that is homeomorphic to a closed, oriented, smooth surface $M$. Then the \(2\)nd homology group \(H_2^{\IC}(X)\) via integral currents is infinite cyclic and admits a generator \(T\) such that 
\(
\phi_{\#}T=\deg(\phi)\cdot \bb{M}
\) 
for every Lipschitz map \(\phi\colon X \to M\). Moreover, \(\norm{T} \leq C \cdot \Haus^2\) for some universal constant \(C\).
\end{theorem}

We emphasize that if $X$ is as in the theorem then in general there do not exist any Lipschitz maps $f\colon U\to X$, defined on an open subset $U\subset \R^2$, such that the image of $f$ has non-trivial Hausdorff $2$-measure. In particular, Theorem~\ref{thm:existence-intcurr-dimension2-intro} is false if the singular Lipschitz homology group \(H_2^{\Lip}(X)\) is considered instead of \(H_2^{\IC}(X)\).
 
\subsection{Applications}\label{sec:applications}
We now proceed by discussing applications of Theorems~\ref{thm:main} and \ref{thm:existence-intcurr-dimension2-intro}. Recently, Z\"ust proved a Lipschitz-volume rigidity result for so-called integral current spaces in \cite{Zuest-23}.  This result together with our Theorems~\ref{thm:main} and \ref{thm:existence-intcurr-dimension2-intro} lets us deduce the following Lipschitz-volume rigidity theorem for metric manifolds.

\begin{theorem}\label{thm:Lipschitz-Volume-rigidity}
Let $X$ be a metric space homeomorphic to a closed, orientable, smooth $n$-manifold. In case $n>2$ assume furthermore that $X$ is linearly locally contractible. If $M$ is a closed, orientable, Riemannian $n$-manifold such that $\Haus^n(X)=\Haus^n(M)$, then every surjective $1$-Lipschitz map from $X$ to $M$ is an isometric homeomorphism.
\end{theorem}

When $n=2$ this recovers the very recent result \cite[Theorem 1.1]{Meier-Ntalampekos-23} in the case that $X$ and $M$ are orientable. The special case that \(n=2\) and \(X\) is also quasiconvex allows for a particularly simple and accessible proof that uses only techniques from \cite{basso2021undistorted, Meier-Wenger, Zuest-23}, see Section~\ref{sec:section8} below. Lipschitz-volume rigidity results are useful in the context of boundary rigidity and uniqueness of minimal fillings of closed Riemannian manifolds \cite{burago-2010, burago-2013}. Variants of Theorem~\ref{thm:Lipschitz-Volume-rigidity} for \(X\) an Alexandrov space, limit RCD space and integral current space, respectively, were obtained in \cite{basso2022filling, perales-2023, li-2015,  li-2014, storm-2006,   Zuest-23}.

Theorem~\ref{thm:main} can also be used to establish geometric properties such as relative isoperimetric inequalities under an additional Ahlfors regularity condition. Recall that a metric space $X$ is called \textit{Ahlfors $n$-regular} if there exists $\alpha\geq 1$ such that 
$$
\alpha^{-1}\cdot r^n\leq \Haus^n(B(x,r))\leq \alpha \cdot r^n
$$ 
for every $x\in X$ and every $r\in(0,\diam X)$. 
The (lower) \textit{Minkowski content} of a subset $E\subset X$ with respect to an open subset $U\subset X$ is defined by 
$$
\mathscr{M}_-(E\, |\ U)= \liminf_{r\searrow 0}\,\frac{\Haus^n(E_r\cap U) - \Haus^n(E\cap U)}{r},
$$ 
where $E_r=\{x\in X: d(x,E)<r\}$ denotes the open $r$-neighborhood of $E$ in $X$.  This provides a substitute for the surface measure of the boundary $\partial E$ of $E$ inside $U$ and is closely related to the perimeter of $E$ inside $U$; see \cite{ambrosio-dimarino-gigli-2017}. For an open ball $B=B(x,r)$ and $\lambda>0$ we denote by $\lambda B$ the open ball $B(x,\lambda r)$ with the same center $x$ and stretched radius $\lambda r$.

\begin{theorem}\label{thm:relative-isoperimetric-inequality}
Suppose \(X\) is a metric space homeomorphic to a closed, orientable, smooth manifold of dimension \(n\geq 2\). If \(X\) is linearly locally contractible and Ahlfors \(n\)-regular, then there exist \(C\), \(\lambda \geq 1\) such that 
\begin{equation}\label{eq:rel-isop-intro}
    \min\bigl\{\Haus^n\big(E\cap B\big), \Haus^n\big(B \setminus E\big) \bigr\} \leq C \cdot\mathscr{M}_-\big( E\, | \, \lambda B\big)^{\,\tfrac{n}{n-1}}
\end{equation}
for every Borel subset \(E\subset X\) and every open ball \(B\subset X\).
\end{theorem}

The constants $C$ and $\lambda$ only depend on the data of $X$, that is, the dimension $n$, the Ahlfors regularity constant, and the linear local contractibility constant.
The proof heavily relies on our main existence result, Theorem~\ref{thm:main} above. In the special case that $X$ is a smooth (Riemannian) manifold, the idea of proof goes back to Gromov and is briefly sketched in \cite[Remark B.20]{semmes-1996}. Notice that in this special case, the existence of a current $T$ as in Theorem~\ref{thm:main} is trivial.

Relative isoperimetric inequalities such as \eqref{eq:rel-isop-intro} are intimately related to Poincar\'e inequalities. Let $(X,d)$ be a metric space equipped with some finite Borel measure $\mu$. A Borel function $g\colon X\to[0,\infty]$ is an upper gradient of a function $u\colon X\to \R$ if 
$$
|u(\gamma(1)) - u(\gamma(0))|\leq \int_\gamma g
$$ 
for every curve $\gamma\colon[0,1]\to X$ of finite length. The metric measure space $(X, d, \mu)$ is said to support a weak $p$-Poincar\'e inequality if there exist $C,\lambda\geq 1$ such that 
$$
\vint_{B} |u-u_B|\,d\mu \leq C\diam(B)\cdot\left(\vint_{\lambda B}g^p\,d\mu\right)^{\frac{1}{p}}
$$ 
for every ball $B\subset X$ and every function-upper-gradient pair $(u,g)$. Here, we have used the notation $u_B=\dashint_{B} u\,d\mu = \mu(B)^{-1}\int_Bu\,d\mu$. 
Poincar\'e inequalities, especially in combination with a doubling condition, provide a suitable setting to build a robust theory of first order calculus in metric measure spaces and they have many geometric and analytic consequences. See for example \cite{Heinonen-Koskela-Shanmugalingam-Tyson-2015} and the many references therein for details.
It is well-known, see e.g.~\cite{Bobkov-Houdre-1997, Kinnunen-Korte-Shanmugalingam-Tuominen-2012, Korte-Lahti-2014}, that relative isoperimetric inequalities imply Poincar\'e inequalities. As a consequence of Theorem~\ref{thm:relative-isoperimetric-inequality}, we therefore obtain the following corollary. 

\begin{corollary}\label{cor:Semmes-theorem-Poincare}
Let $X$ be a metric space homeomorphic to a closed, orientable, smooth manifold of dimension $n\geq 2$. If $X$ is Ahlfors $n$-regular and linearly locally contractible then it supports a weak $1$-Poincar\'e inequality.
\end{corollary}

This result was originally proved by Semmes in his seminal article \cite{semmes-1996}. Our approach provides a short proof of Semmes' result, using completely different methods.

In all our results above the condition that $X$ be homeomorphic to a closed, oriented, smooth \(n\)-manifold can be weakened to the assumption that \(X\) be homeomorphic to a closed, oriented, Lipschitz manifold \(N\) equipped with a compatible metric. Thus, when $n\not=4$ it suffices to assume that $N$ be a topological manifold, since by Sullivan's result  \cite{sullivan-1979}, every such manifold admits a (unique) Lipschitz structure, see also \cite{tukia-vaisala-1981}.

\subsection{Strategy of proofs} We finally sketch some of the proofs of our main results and  first explain how to construct a current $T$ as in Theorem~\ref{thm:main}. Roughly speaking, the idea is to equip the rectifiable part of $X$ with a suitable orientation and define $T$ by integration over this set. In order to do this, decompose $X$ into its $n$-rectifiable part $E$ and its purely $n$-unrectifiable part $S$. Up to a negligible set, $E$ can be written as a countable union of pairwise disjoint images of bi-Lipschitz maps $\rho_i\colon K_i\to X$, defined on compact sets $K_i\subset\R^n$. The idea is now to define $T$ as the sum of currents $$
T_i(f,\pi) = \int_{K_i}\theta_i\cdot (f\circ\rho_i) \, \det(D(\pi\circ\rho_i))\,d\mathscr{L}^n
$$ 
for suitable choices of functions $\theta_i\colon K_i\to\R$ with $\abs{\theta_i}=1$ almost everywhere. For any such “orientation" $\theta_i$ this defines an integer rectifiable $n$-current on $X$ whose mass measure $\norm{T}$ is bounded, up to a constant, by the Hausdorff \(n\)-measure. We want to choose the $\theta_i$ in such a way that $T$ has zero boundary. Notice that $\rho_i$ can wildly switch orientation and that $K_i$ need not have finite perimeter, so $T_i$ can have infinite boundary mass. We use the linear local contractibility of $X$ in order to find a suitable  continuous extension of $\rho_i$ to an open neighborhood of $K_i$, see Lemma~\ref{lem:good-continuous-extension-LLC}, and then define $\theta_i$ to be the local degree of this extension at almost every point of $K_i$. Applying topological degree theory and a recent 
result of Bate \cite{bate-2020} about the Hausdorff measure of Lipschitz images of purely unrectifiable sets we show that $T$ has zero boundary, that is, $T(1,\pi)=0$. Indeed, by Bate's theorem we may assume after approximation that $\pi$ maps the purely unrectifiable part $S$ of \(X\) to a negligible set and thus for almost every $y\in\R^n$ the preimage $\pi^{-1}(y)$ does not intersect $S$. Using that for almost every $y\in\R^n$ the sum of the local degrees $\deg(\pi, x)$ for $x\in\pi^{-1}(y)$ equals the topological degree of $\pi$ (which is zero because $\R^n$ is non-compact) we can show that $T(1,\pi)=0$ and hence $\partial T=0$. The same reasoning also shows that $\varphi_\#T = \deg(\varphi)\cdot\bb{M}$ for every Lipschitz map $\varphi\colon X\to M$. Since \(X\) and \(M\) are homeomorphic and \(M\) is Riemannian, an easy Lipschitz approximation result shows that there exists a Lipschitz map \(\phi\colon X\to M\) with \(\deg(\phi)\neq 0\). Hence, we find that \(T\neq 0\) and so \(X\) admits a non-trivial integral \(n\)-cycle. 
In case of metric surfaces we do not need to impose the linear locally contractability condition to ensure the existence of such a cycle. The main reason for this is that in dimension \(2\) there are strong uniformization results available \cite{ Meier-Wenger, Ntalampekos-Romney-22, Ntalampekos-Romney-21}. These together with a straightforward pushforward construction for currents under Sobolev maps are sufficient to construct the current \(T\) when $n=2$. This is done in Section~\ref{sec:2-dim-existence}, where we prove Theorem~\ref{thm:existence-intcurr-dimension2-intro}.

We finally outline the proof of Theorem~\ref{thm:relative-isoperimetric-inequality} and content ourselves to explaining the ideas behind the isoperimetric inequality 
\begin{equation}\label{eq:isop-ineq-proof-sketch-intro} \Haus^n(E)\leq C\cdot \mathscr{M}_-(E\, | \, X)^{\frac{n}{n-1}}
\end{equation} 
for Borel sets $E\subset X$ of diameter at most  $\varepsilon \diam X$ for some small $\varepsilon>0$ depending only on the data of $X$. The proof of the general case uses the same ingredients but is more involved, see Section~\ref{sec:relative-isoperimetric-inequality}. When $X$ and $E$ are sufficiently smooth the idea of proof is sketched by Semmes in \cite[Remark B.20]{semmes-1996}, where it is attributed to Gromov. In this special case, one can interpret $E$ as a (Lipschitz) $n$-chain and the idea is to use a minimal filling of the boundary of $E$, interpreted as an $(n-1)$-cycle, in $\ell^\infty$ and compare the volume of $E$ with the mass of the minimal filling.

In our general setting, where there might not be any non-trivial Lipschitz $n$-chains in $X$, we can use the current $T$ from Theorem~\ref{thm:main} instead. We thus consider the restriction $T'=T\on E_r$ of $T$ to the open $r$-neighborhood of $E$ for some small $r>0$. Choosing $r$ suitably we may achieve that $T'$ is an integral current and satisfies $\mass(\partial T')\preceq \mathscr{M}_-(E\,|\, X)$, where $\preceq$ means inequality up to a constant depending only on the data of $X$. Now, view $X$ as a subset of its injective hull $E(X)$ and let $S\in \bI_n(E(X))$ be a minimal filling of $\partial T'$ in $E(X)$. The definition of \(E(X)\) can be found in Section~\ref{sec:metric-notions}. The spaces \(E(X)\) and \(\ell^\infty\) have similar properties, but we prefer to work with the injective hull because in our situation $E(X)$ is also compact. Such a filling \(S\) of \(\partial T'\) exists and satisfies a Euclidean isoperimetric inequality $$\mass(S)\preceq \mass(\partial S)^{\frac{n}{n-1}}\preceq \mathscr{M}_-(E\, |\, X)^{\frac{n}{n-1}}.$$ Moreover, there is a lower bound of the form 
\begin{equation}\label{eq:lower-bound-area-min-intro}
\norm{S}(B(y,s))\succeq s^n
\end{equation}
for any ball centered on the support $\spt S$ with radius  $0<s\leq d(y,\spt(\partial S))$. 

The proof of \eqref{eq:isop-ineq-proof-sketch-intro} is complete if we can show that $\Haus^n(E)\preceq\mass(S)$. In order to establish such a bound we first use the fact that $X$ is a doubling metric space as well as linearly locally contractible, to construct a continuous retraction $\pi\colon N_{R}(X) \to X$ from the $R$-neighborhood of $X$ in $E(X)$, where $R=\lambda \diam X$, with the property that $d(\pi(y),x)\leq C'\cdot d(y,x)$ for all $y\in N_R(X)$ and all $x\in X$. The constants \(C'\), \(\lambda>0\) depend only on the data of \(X\); see Proposition~\ref{prop:good-retraction-property}. It is not difficult to see that we may assume $\spt S\subset N_R(X)$. We now claim that $E\subset \pi(\spt S)$ and that $\Haus^n(\pi(\spt S))\preceq \mass(S)$, from which the desired bound follows. The first claim is a consequence of the fact that the push-forward of the integral $n$-cycle $Q = T'-S$ under a Lipschitz map $\varphi$ closely approximating $\rho\circ\pi$ for some homeomorphism $\rho\colon X\to M$ must be the zero current in $M$ by the constancy theorem. The proof of the second claim relies on the observation that for $y\in \spt S$ and $r=d(y, \spt(\partial S))$ we have $\pi(B(y,r))\subset B(\pi(y), 3C'r)$ and hence, together with the Ahlfors regularity and \eqref{eq:lower-bound-area-min-intro}, we obtain $$\Haus^n(\pi(B(y,r))\preceq \Haus^n(B(\pi(y), 3C'r))\preceq r^n\preceq \norm{S}(B(y,r)).$$ The second claim follows from this by using a suitable covering of $\spt S$ by balls of bounded multiplicity. This finishes the outline of the proof of \eqref{eq:isop-ineq-proof-sketch-intro}.

\subsection{Structure of the paper}
The paper is organized as follows. In Section~\ref{sec:prelims} we fix notation and recall notions concerning topological degree theory, Nagata dimension, and Sobolev maps with values in metric spaces. We also recall the necessary background regarding Bate's result \cite{bate-2020}. In Section~\ref{sec:currents} we state the basic notions from the theory of metric currents needed for this paper. We also briefly describe the main properties of minimal fillings in injective metric spaces needed in Sections~\ref{sec:main-result} and \ref{sec:relative-isoperimetric-inequality}. In the next section, as a preparatory result for the proof of Theorem~\ref{thm:main}, we prove a generalized version of a theorem of Kinneberg \cite{Kinneberg-2018}. This result does not use the theory of currents. In Section~\ref{sec:main-result} we are concerned with the proof of Theorem~\ref{thm:main}. There we also establish Corollary~\ref{cor:recitifiability}, our rectifiability criterion for metric manifolds. In Section~\ref{sec:2-dim-existence} we prove Theorem~\ref{thm:existence-intcurr-dimension2-intro}. Using Theorem~\ref{thm:main}, we prove in Section~\ref{sec:relative-isoperimetric-inequality} the relative isoperimetric inequality stated in Theorem~\ref{thm:relative-isoperimetric-inequality}. The validity of the weak \(1\)-Poincar\'e inequality, Corollary~\ref{cor:Semmes-theorem-Poincare}, is proved at the end of the same section. Finally, in Section~\ref{sec:section8}, we combine Theorems~\ref{thm:main} and \ref{thm:existence-intcurr-dimension2-intro} with Z\"ust's result \cite{Zuest-23} to derive Theorem~\ref{thm:Lipschitz-Volume-rigidity}. There we also give an alternative proof for the special case of Theorem~\ref{thm:existence-intcurr-dimension2-intro} when \(X\) is a quasiconvex metric surface. 



\section{Preliminaries}\label{sec:prelims}

\subsection{Metric notions}\label{sec:metric-notions}
Let \(X=(X,d)\) be a metric space. We use \(B(x, r)=\{ x'\in X : d(x, x') < r\} \) to denote the open ball with center \(x\in X\) and radius \(r>0\). We say that \(X\) is \textit{\(\Lambda\)-linearly locally contractible} if every ball \(B(x, r)\) is contractible in \(B(x, \Lambda r)\) for every \(r\in \big(0, \Lambda^{-1} \diam X\big)\). Given \(A\), \(B\subset X\), we write
\[
d(A, B)=\inf\big\{ d(a, b) : a\in A \text{ and } b\in B\big\}
\]
for the infimal distance between \(A\) and \(B\). Notice that, in particular, \(d(A, \varnothing)=\infty\). Further, for any \(A\subset X\),
\[
N_r^X(A)=\{ x\in X : d(x, A) < r\}
\]
denotes the open \(r\)-neighborhood of \(A\). If the ambient space \(X\) is clear from the context, we will often write \(N_r(A)\) instead of \(N_r^X(A)\). Moreover, we use the convention that \(N_r(A)=\varnothing\) if \(r\leq 0\). Occasionally, we also use the notation \(\bar{N}_r(A)\) to denote the union of \(N_r(A)\) and \(\{ x\in X : d(x, A)=r\}\). We say that \(A\subset X\) is an \textit{\(r\)-net} if \(N_r(A)=X\) and \(d(a, a')\geq r\) for all distinct \(a\), \(a'\in A\).

We say that a map \(f\colon X \to Y\) between metric spaces is \textit{\(L\)-Lipschitz} if \(d(f(x),f(y))\leq L d(x, y)\) for all \(x\), \(y\in X\).  The smallest \(L\geq 0\) such that \(f\) is \(L\)-Lipschitz is denoted \(\Lip(f)\). If \(f\) is injective such that both \(f\) and \(f^{-1}\) are \(L\)-Lipschitz, then we say that \(f\) is \textit{\(L\)-bi-Lipschitz}. Given two maps \(f, g\colon X \to Y\) we write
\[
d(f, g)=\sup\big\{ d(f(x), g(x)) : x\in X\big\}
\]
for the uniform distance between \(f\) and \(g\). We will often use the following easy-to-check Lipschitz approximation result.

\begin{lemma}\label{lem:lip-approx-easy}
Let \(f\colon X \to Y\) be a continuous map form a compact metric space \(X\) to a separable metric space \(Y\) which is an absolute Lipschitz neighborhood retract. Then for every \(\epsilon>0\) there exists a Lipschitz map \(g \colon X \to Y\) with \(d(f, g)<\epsilon\). 
\end{lemma}

Recall that \(Y\) is said to be an \textit{absolute Lipschitz neighborhood retract} if there exists \(C>0\) such that whenever \(Y\subset Y_0\) for some metric space \(Y_0\), then there exists a \(C\)-Lipschitz retraction \(R\colon U \to Y\) where \(U\subset Y_0\) is an open neighborhood of \(Y\). Every closed Riemannian \(n\)-manifold is an absolute Lipschitz neighborhood retract, see e.g. \cite[Theorem~3.1]{hohti-1993}.

\begin{proof}[Proof of Lemma~\ref{lem:lip-approx-easy}]
If \(Y=\R\), this follows directly from \cite[Lemma~2.4]{semmes-1996}. Since \(Y\) is separable we may suppose that \(Y\subset \ell^\infty\) via Fréchet's embedding. Hence, the general case can also be proved by applying \cite[Lemma~2.4]{semmes-1996} to each coordinate function.
\end{proof}

We will also use some basic facts about injective metric spaces and injective hulls. A metric space \(Y\) is called  \textit{injective} if for all pairs \(A\subset B\) of metric spaces, any \(1\)-Lipschitz map \(f\colon A \to Y\) admits a \(1\)-Lipschitz extension \(\bar{f}\colon B\to Y\). For every metric space \(X\) there is an injective space \(E(X)\), called the \textit{injective hull} of \(X\), such that \(X\subset E(X)\) and any isometric embedding \(X \to Y\) into an injective space \(Y\) admits an isometric extension \(E(X) \to Y\). Hence, roughly speaking, \(E(X)\) is the smallest injective metric space containing \(X\). This important result is due to Isbell \cite{isbell-1964} and was later rediscovered by Dress \cite{dress-1984}. Injective metric spaces are always complete and if \(X\) is compact, then \(E(X)\) is compact as well, see \cite[Section 3]{lang-2013}.

The following topological result from \cite{Petersen93} will be of importance in Section~\ref{sec:lower-bound-Hausdorff}. Let $\varepsilon>0$. A map $f\colon X\to Y$ is called $\varepsilon$-continuous if there exists $\delta>0$ such that $d(f(x), f(y))<\varepsilon$ for all $x, y\in X$ with $d(x,y)<\delta$.

\begin{proposition}\label{prop:eps-cont-extension}
Let $X$, $Y$ be compact metric spaces such that $X$ has topological dimension at most $n$ and $Y$ is linearly locally contractible. Let $A$ be a (possibly empty) subset of $X$. There exists $Q\geq 1$ such that if $f\colon X\to Y$ is $\varepsilon$-continuous on $X$ with $\varepsilon<Q^{-1}\diam Y$ and $f$ is continuous at every point of $A$ then there is a continuous map $g\colon X\to Y$ which coincides with $f$ on $A$ and satisfies $d(f,g)<Q\cdot\varepsilon$.
\end{proposition}

Here, $Q$ only depends on $n$ and the linear local contractibility constant. We note a particular consequence of the proposition: if two maps $f_0, f_1\colon X\to Y$ are continuous and satisfy $d(f_0,f_1)<\varepsilon$ then they are homotopic via a homotopy $H\colon X\times[0,1]\to Y$ subject to 
$$
d(H(x,t), f_0(x))<(1+Q)\cdot\varepsilon
$$ 
for all $x\in X$ and $t\in[0,1]$.

\subsection{Orientation and degree}
We briefly recall the definitions of orientation on a topological manifold and the degree of continuous maps between oriented topological manifolds. For details we refer to \cite{dold-1980}. Let \(X\) be a topological \(n\)-manifold not necessarily compact. As mentioned already, throughout this article all manifolds are assumed to be connected and without boundary. For every \(x\in X\) the relative singular homology group $H_n(X, X\setminus x)$ taken with integer coefficients is isomorphic to \(\Z\). An \textit{orientation} of $X$ is a choice of generator $o_x\in H_n(X, X\setminus \hspace{-0.1em}x)$ for each $x\in X$ such that the following continuity property holds. For every $x\in X$ there is an open neighborhood $U\subset X$ of $x$ and $z\in H_n(X, X\setminus \hspace{-0.1em}U)$ such that the homomorphism from $H_n(X, X\setminus \hspace{-0.1em}U)$ to $H_n(X, X\setminus \hspace{-0.1em}x)$ induced by inclusion sends $z$ to $o_x$. It can be shown that if $\{o_x\}$ is an orientation of $X$ then for every compact set $K\subset X$ there is a unique homology class $o_K\in H_n(X, X\setminus \hspace{-0.1em}K)$ such that for every $x\in K$ the homomorphism $H_n(X, X\setminus \hspace{-0.1em}K)\to H_n(X, X\setminus \hspace{-0.1em}x)$ induced by inclusion maps $o_K$ to $o_x$. For compact $X$ the homology class $o_X$ is denoted $[X]$ and called the \textit{fundamental class} of $X$. Let $f\colon X\to Y$ be a continuous map between oriented topological $n$-manifolds. Let $K\subset Y$ be a non-empty compact connected set such that $f^{-1}(K)$ is compact. Then  \(f_\ast \colon H_n(X, X \setminus \hspace{-0.1em}f^{-1}(K)) \to H_n(Y, Y\setminus \hspace{-0.1em}K)\) sends \(o_{f^{-1}(K)}\) to an integer multiple of \(o_K\). This integer is denoted by \(\deg_K f\) and called the \textit{degree of \(f\) over \(K\)}. It holds that \(\deg_K (f)=\deg_{\hspace{0.1em}\{y\}}(f)\) for every \(y\in K\).  In particular, if $X$ is compact then \(\deg_K f\) is the same for all non-empty compact connected \(K\subset Y\). This number is denoted $\deg f$ and called the \textit{degree of \(f\)}.  If both \(X\) and \(Y\) are compact, then \(f_\ast([X])=\deg (f)\cdot [Y]\). Moreover, if \(X\) is compact and \(Y\) is non-compact, then \(\deg f=0\), as \(f(X)\) is a proper subset of \(Y\). We remark that the degree of \(f\) is homotopy invariant, that is, if \(g\colon X \to Y\) is a continuous map homotopic to \(f\), then \(\deg(f)=\deg(g)\).

\subsection{Nagata dimension}
A covering of a metric space \(X\) has \textit{\(s\)-multiplicity} at most \(n\) if every subset of \(X\) with diameter less than \(s\)  meets at most \(n\) members of the covering. We emphasize that we consider coverings  by arbitrary subsets and not only by open subsets. The following definition can be thought of as a quantitative version of topological dimension.

\begin{definition}
Let \(n\) be a non-negative integer and \(c\geq 0\). A metric space \(X\) satisfies \(\text{Nagata}(n,c)\) if it admits for every \(s>0\) a covering with \(s\)-multiplicity at most \(n+1\) such that \(\diam B \leq cs\) for every member \(B\) of the covering.
\end{definition}

The smallest non-negative integer \(n\) such that \(X\) satisfies \(\text{Nagata}(n,c)\) for some \(c\geq 0\) is denoted by \(\dim_N(X)\) and called the \textit{Nagata dimension} of \(X\). This definition goes back to Assouad \cite{assouad-1982}, building on earlier work of Nagata \cite{nagata-1958}. Every doubling metric space satisfies \(\text{Nagata}(N, 2)\), where \(N\) depends only on the doubling constant. See \cite[Lemma~2.3]{lang-2005}. Many other basic properties of the Nagata dimension can be found in the article \cite{lang-2005} by Lang and Schlichenmaier.

If \(A \subset X\) is of finite Nagata dimension, then \(X \setminus A\) admits some kind of Whitney covering, see \cite[Theorem 5.2]{lang-2005}. In our proof of Theorem~\ref{thm:relative-isoperimetric-inequality} in Section~\ref{sec:relative-isoperimetric-inequality}, we need the following simple lemma which states that \(X\setminus A\) also has such a covering by open balls.

\begin{lemma}\label{lem:Nagata-cover}
Let \(X\) be a metric space and \(A\subset X\) a non-empty closed subset of finite Nagata dimension. Then there exist \(a\in (0,1)\), \(b\), \(L\geq 1\) and \(F\subset X\setminus A\) such that
\[
 X\setminus A \subset \bigcup_{x\in F} B(x, b\cdot r_x),
\]
where \(r_x=d(x, A)\), and \(\{  B(x, a\cdot r_x)\}_{x\in F}\) has multiplicity at most $L$. The constants \(a\), \(b\), \(L\) depend only on the data of \(A\).
\end{lemma}

We remark that if \(A\) satisfies \(\text{Nagata}(n, c)\), then the proof of Lemma~\ref{lem:Nagata-cover} shows that one can take \(a=1/4\), \(b=10c\) and \(L=3(n+1)\). 

\begin{proof}
We abbreviate \(r=2\), \(a=1/4\) and fix \(k\in \Z\). We set
\[
R=R_k=\{ x\in X : r^k \leq d(x, A) < r^{k+1}\}
\]
and let \(W\subset R\) be an \(a r^k\)-net. Clearly, there exists a map \(\rho \colon W \to A\) such that \(d(\rho(w), w)\leq r^{k+1}\) for all \(w\in W\). Suppose that \(A\) satisfies \(\text{Nagata}(n, c)\), for some \(n\in \N\) and \(c\geq 1\), and let \((A_i)_{i\in I}\) be a \(cs\)-bounded covering of \(A\) with \(s\)-multiplicity at most \( n+1\), where \(s=5 r^k\). Then, by construction, the sets
\[
B_i=\big\{ x\in X\setminus A : \text{there is \(w\in W\) with \(d(w, x)\leq a r^k\) and \(\rho(w)\in A_i\)} \big\}
\]
cover \(R\) and \(\diam B_i \leq 2(a r^k+r^{k+1})+5 c r^k < 10 c r^k\). Let \(J\subset I\) consist of those \(i\in I\) for which \(B_i \cap W\neq \varnothing\) and select \(w_j\in B_j\cap W\) for every \(j\in J\).
It follows that \(\{ B(w_j, br_j)\}_{j\in J}\) covers \(R\), where \(b=10c\) and \(r_j=d(w_j, A)\). Moreover, any \(x\in X \setminus A\) meets at most \((n+1)\)  members of \(\{ B(w_j, ar_j)\}_{j\in J} \). Indeed, letting \(M=\big\{ \rho(w_j) : B(w_j, ar_j) \cap \{x\} \neq \varnothing\big\}\), we find that
\[
\diam M \leq 2( r^{k+1}+ar^{k+1})=5 r^k
\]
and consequently \(M\) meets at most \((n+1)\) members of \((A_i)_{i\in I}\). This implies that \(x\) meets at most \((n+1)\) sets of \(\mathcal{B}_k=\{ B(w_j, a r_j)\}_{j\in J}\). Since each ball \(B(w_j, a r_j)\) is contained in \(R_{k-1}\cup R_k \cup R_{k+1}\), it follows that any \(x\in X\setminus A\) meets at most \(3(n+1)\) members of \(\mathcal{B}=\bigcup_{k\in \Z} \mathcal{B}_k\). 
\end{proof}

\subsection{Perturbations of Lipschitz functions} Let \(X\) be a complete metric space. We let \(\Haus^n\) denote the Hausdorff \(n\)-measure on \(X\). We normalize \(\Haus^n\) such that it is equal to the Lebesgue measure \(\mathscr{L}^n\) on \(\R^n\). We say that a \(\Haus^n\)-measurable set \(E\subset X\) is \textit{\(n\)-rectifiable} if there exist compact subsets \(K_i\subset \R^n\) and bi-Lipschitz maps \(\rho_i\colon K_i \to X\) such that  the images \(\rho_i(K_i)\) are pairwise disjoint and 
\[
\Haus^n\Big(E\setminus \bigcup_{i\in \N} \rho_i(K_i)\Big)=0.
\]
On the other hand, a \(\Haus^n\)-measurable set \(S\subset X\) is called \textit{purely \(n\)-unrectifiable} if \(\Haus^n(E\cap S)=0\) for every \(n\)-rectifiable \(E\subset X\). 
For any \(A\subset X\) and \(x\in X\), we let
\[
\Theta_{\ast n}(A, x)=\liminf_{r\searrow 0} \frac{\Haus^n(A \cap B(x, r))}{\omega_n r^n} 
\]
denote the \textit{lower density} of \(A\) at \(x\). Here, \(\omega_n\) equals the Lebesgue measure of the Euclidean unit \(n\)-ball. We remark that if $\Haus^n(X)<\infty$ and $\Theta_{\ast n}(X, x)>0$ for $\Haus^n$-almost every $x\in X$ then for every $\Haus^n$-measurable subset $A\subset X$ we also have $\Theta_{\ast n}(A,x)>0$ for $\Haus^n$-almost every $a\in A$; see \cite[2.10.19(4)]{Federer-GMT-1969}. This observation will be useful in the sequel when applying the following deep result of Bate which states that a generic bounded \(1\)-Lipschitz map to \(\R^m\) maps purely \(n\)-unrectifiable sets to \(n\)-negligible sets. 

\begin{theorem}{\normalfont (Bate \cite[Theorem~1.1]{bate-2020})}\label{thm:bates-thm}
 Let \(X\) be a complete metric space and \(S\subset X\) be purely \(n\)-unrectifiable with \(\Haus^n(S)<\infty\) and \(\Theta_{\ast n}(S, x)>0\) for \(\Haus^n\)-almost every \(x\in S\). Then for every \(m\in \N\), the set of all \(f\in \Lip_{1}(X, \R^m)\) with \(\Haus^{n}(f(S))=0\) is residual. 
\end{theorem}
  
Here, \(\Lip_{1}(X, \R^m)\) denotes the set of all bounded \(1\)-Lipschitz functions equipped with the supremum norm. Since this is a complete metric space, it follows from the Baire category theorem that every residual subset is dense.

\subsection{Sobolev maps to metric spaces}
In the proof of Theorem~\ref{thm:existence-intcurr-dimension2-intro} about the existence of integral currents in metric surfaces we will use Sobolev maps from a Riemannian manifold (of dimension $2$) to a metric space. There are several equivalent definitions of Sobolev maps from Euclidean space or a Riemannian manifold to a complete metric space, see \cite{Ambrosio-BV-90, Korevaar-Schoen-93, Reshetnyak-97, Heinonen-Koskela-Shanmugalingam-Tyson-2015}.

Let $M$ be a closed smooth $n$-dimensional manifold. We fix a Riemannian metric $g$ on $M$. Let furthermore $(X,d)$ be a complete metric space. A measurable and essentially separably valued map $\rho\colon M\to X$ is said to belong to the Sobolev space $W^{1,p}(M,X)$ if for every $x\in X$ the function $$\rho_x(z):= d(x, \rho(z))$$ belongs to the classical Sobolev space $W^{1,p}(M)$ and there exists $h\in L^p(M)$ such that for all $x\in X$ we have $|\nabla \rho_x|\leq h$ almost everywhere. Here, $|\nabla \rho_x|$ denotes the length (with respect to the Riemannian metric) of the weak gradient of $\rho_x$.

Every $\rho\in W^{1,p}(M, X)$ has an approximate metric derivative at almost every $z\in M$, that is, there exists a seminorm $\apmd\rho_z$ on $T_zM$ satisfying 
$$
\ap\lim_{v\to 0} \frac{d(\rho(\exp_z(v)), \rho(z)) - \apmd\rho_z(v)}{|v|_g} = 0.
$$ 
Here, $\ap\lim$ denotes the approximate limit, and $\apmd\rho$ is called the \textit{approximate metric derivative}. If $\rho$ is Lipschitz then the honest limit exists and, in this case, the seminorm is denoted $\md\rho$ and called the metric derivative. If $\rho\colon A\to X$ is a Lipschitz map defined on a measurable subset of $M$, then one can still make sense of the metric derivative of $\rho$ at almost every point of $A$ by viewing $\rho$ as a map to the injective hull $E(X)$ and extending it to a Lipschitz map defined on all of $M$.

We will need two different notions of parametrized volumes for Lipschitz and Sobolev maps, the parametrized Gromov mass$*$ volume (also called Benson's definition of volume) and the parametrized Hausdorff volume. In order to define these we introduce two Jacobians of a seminorm $s$ on $\R^n$. If $s$ defines a norm on $\R^n$ then we set $\mathbf{J}(s)=\Haus^n_{(\R^n, s)}([0,1]^n)$ and $\mathbf{J}^*(s) = 2^n/\mathscr{L}^n(P)$, where $\mathscr{L}^n(P)$ is the Lebesgue measure of the parallelepiped of smallest volume containing the unit ball with respect to $s$. If $s$ is degenerate, then we set $\mathbf{J}(s)=\mathbf{J}^*(s) =0$. Notice that $\mathbf{J}(s)$ and $\mathbf{J}^*(s)$ are comparable up to a constant only depending on $n$. The (parametrized) Hausdorff volume of a map $\rho\in W^{1,n}(M, X)$ is 
$$
\vol(\rho):= \int_M\mathbf{J}(\apmd\rho)\,d\hspace{-0.14em}\Haus^n.
$$ 
The (parametrized) Gromov mass$*$ volume $\vol^*(\rho)$ is defined analogously using $\mathbf{J}^*$ instead of $\mathbf{J}$. We define the volume of $\rho$ restricted to a measurable subset of $M$ analogously. For other natural definitions of volumes considered in Finsler geometry we refer to the excellent survey article \cite{paiva-2004}.


\section{Background on currents}\label{sec:currents}

In this section we review some basic aspects of the theory of currents in metric spaces that are needed for the proofs of our main results. For details we refer to \cite{ambrosio-kirchheim-2000} or \cite{lang-currents-2011}. Throughout this section, let \(X\) be a complete metric space. 

\subsection{Metric currents} For any \(k\geq 0\) we let \(\mathcal{D}^k(X)\) denote the set consisting of all tuples \((f, \pi_1, \ldots, \pi_k)\) of Lipschitz functions \(f\colon X\to \R\) and \(\pi_i\colon X\to \R\), for \(i=1, \ldots, k\),  with \(f\) bounded. 

\begin{definition}\label{def:metric-current}
A multilinear map \(T\colon \mathcal{D}^k(X)\to \R\) is called metric \(k\)-current (of finite mass) if the following holds:
\begin{enumerate}
    \item(continuity) If \(\pi_i^{j}\) converges pointwise to \(\pi_i\) for every \(i=1, \ldots, k\), and \(\Lip \pi_i^{j}< C\) for some uniform constant \(C>0\), then 
    \[
    T(f, \pi_1^{j}, \ldots, \pi_k^{j})\to T(f, \pi_1, \ldots, \pi_k)
    \] as \(j\to \infty\);
    \item(locality) \(T(f, \pi_1, \ldots, \pi_k)=0\) if for some \(i\in\{1, \ldots, k\}\) the function \(\pi_i\) is constant on \(\{ x\in X : f(x) \neq 0\}\);
    \item(finite mass) There exists a finite Borel measure \(\mu\) on \(X\) such that
    \begin{equation}\label{eq:mass-inequality}
    \abs{T(f, \pi_1, \ldots, \pi_k)} \leq \prod_{i=1}^k \Lip \pi_i \int_{X} \abs{f(x)} \, d\mu(x)
    \end{equation}
    for all \((f, \pi_1, \ldots, \pi_k)\in \mathcal{D}^k(X)\).
    \end{enumerate}
\end{definition}
We call the minimal measure \(\mu\) satisfying \eqref{eq:mass-inequality} the mass measure of \(T\) and denote it by \(\norm{T}\). The support of \(T\) is defined as
\[
\spt T=\big\{ x\in X : \norm{T}(B(x, r)) >0 \text{ for all } r>0\big\}.
\] 
Notice that \(\norm{T}\) is concentrated on \(\spt T\), that is, \(\norm{T}(B)=\norm{T}(B\cap \spt T)\) for any Borel subset \(B\subset X\). We denote by \(\mass_k(X)\) the vector space of \(k\)-currents on \(X\). We set \(\mass(T)=\norm{T}(X)\) and call it the mass of \(T\). One can show that \(\mass_k(X)\) equipped with the mass norm is a Banach space. We remark that one always has \(\mass_k(X)=\{0\}\) if \(k> \dim_N(X)\); see \cite[Proposition
2.5]{zuest-2011} for more details. 

In the following, we recall the definitions of the restriction, push-forward, and boundary operators on \(\mass_k(X)\).   
Since bounded Lipschitz functions are dense in \(L^1(X, \norm{T})\), it follows that \(T\) admits a unique extension to a functional on tuples \((h, \pi_1, \ldots, \pi_k)\) with \(h\in L^1(X, \norm{T})\).  In particular, for any Borel set \(B\subset X\), the restriction \(T\on B \colon \mathcal{D}^k(X)\to \R\) defined by
\[
(T\on B)(f, \pi_1\ldots, \pi_k)=T(\mathbbm{1}_B \cdot f, \pi_1, \ldots, \pi_k)
\]
is well-defined and it is not difficult to show that it is a \(k\)-current. One has that \(\norm{T\on B}=\norm{T}\on B\). 

Let $Y$ be another complete metric space and $\varphi\colon X\to Y$ a Lipschitz map. The push-forward of $T\in\mass_k(X)$ under $\varphi$ is the element $\varphi_\#T\in\mass_k(Y)$ given by $$\varphi_\# T(f, \pi_1, \ldots, \pi_k)=T(f \circ \varphi, \, \pi_1 \circ \varphi, \ldots, \, \pi_k \circ \varphi)$$
for all \((f, \pi_1, \ldots, \pi_k)\in \mathcal{D}^k(Y)\). It is not difficult to show that \((\varphi_\# T)\on B=\varphi_\#(T\on \varphi^{-1}(B))\) for any Borel set \(B\subset Y\); moreover, $\spt(\varphi_\#T)$ is contained in the closure of $\varphi(\spt T)$ and \(\mass(\varphi_\# T)\leq (\Lip \varphi)^k \mass(T)\). Finally, it is easy to show that one can make sense of the push-forward also when $\varphi$ is only defined on $\spt T$ by using arbitrary Lipschitz extensions of $f\circ\varphi$ and $\pi\circ\varphi$.

For \(k\geq 1\) the boundary of \(T\in \mass_k(X)\) is defined by
\[
(\partial T)(f, \pi_1, \ldots, \pi_{k-1})=T(1, f, \pi_1\ldots, \pi_{k-1})
\]
for all \((f, \pi_1, \ldots, \pi_{k-1})\in \mathcal{D}^{k-1}(X)\). This defines a \((k-1)\)-multilinear functional that satisfies the continuity and locality axioms in Definition~\ref{def:metric-current}.
If \(\partial T\in \mass_{k-1}(X)\), then \(T\) is called a \textit{normal} current. We abbreviate \(\bN_0(X)=\mass_0(X)\) and for \(k\geq 1\) we let \(\bN_k(X)\) denote the set of all normal \(k\)-currents on \(X\). These vector spaces become Banach spaces when they are equipped with the norm \(\bN(T)=\mass(T)+\mass(\partial T)\) with the convention that \(\bN(T)=\mass(T)\) if \(T\in \bN_{0}(X)\). 

\subsection{Integer rectifiable and integral currents} In what follows, we introduce integral currents, which are the main class of currents considered in this article. Every \(\theta\in L^1(\R^k)\) induces a \(k\)-current \(\bb{\theta}\) on \(\R^k\) defined by
\[
\bb{\theta}(f, \pi_1, \ldots, \pi_k)=\int_{\R^k} \theta f \det\big( D\pi\big)\, d\mathscr{L}^k
\]
for all \((f,\pi)=(f, \pi_1, \ldots, \pi_k)\in \mathcal{D}^k(\R^k)\). If \(\theta\) has bounded variation then \(\bb{\theta}\in \bN_k(\R^k)\). These currents can be used as building blocks for integer rectifiable and integral currents.
 
\begin{definition}
A current $T\in\bM_k(X)$ is said to be an \textit{integer rectifiable} current if there are compact subsets \(K_i\subset \R^k\), functions \(\theta_i\in L^1(\R^k, \Z)\) with \(\spt \theta_i \subset K_i\), and bi-Lipschitz maps \(\rho_i \colon K_i \to X\) such that
\[
T=\sum_{i\in \N} \rho_{i\#}\bb{\theta_i} \quad \text{ and } \quad \mass(T)=\sum_{i\in \N} \mass(\rho_{i\#}\bb{\theta_i}).
\]
Integer rectifiable  normal currents are called \textit{integral} currents.  The set of all integral \(k\)-currents in \(X\) is denoted by \(\bI_k(X)\). 
\end{definition}
 For \(T\in \bI_k(X)\), it follows that \(\norm{T}\) is concentrated on the \(k\)-rectifiable set
\begin{equation}\label{eq:characteristic-set}
\set T= \Big\{ x\in X : \Theta_{\ast k}(\norm{T}, x) >0 \Big\},
\end{equation}
which is called the \textit{charactersistic set} of \(T\). If \(\varphi\colon X\to Y\) is a Lipschitz map, then \(\varphi_\# T\in \bI_k(Y)\) and \(\partial(\varphi_\# T)=\varphi_{\#}(\partial T)\) whenever \(T\in \bI_k(X)\). In particular, any Lipschitz \(k\)-chain induces an integral \(k\)-current. We remark that \(\bI_k(X)\) is an additive abelian subgroup of \(\bN_k(X)\). Hence, the boundary-rectifiability theorem of Ambrosio--Kirchheim implies that 
\[
\dotsm \overset{\partial_{k+1}}{\longrightarrow} \bI_k(X) \overset{\partial_{k}}{\longrightarrow} \bI_{k-1}(X) \overset{\partial_{k-1}}{\longrightarrow} \dotsm \overset{\partial_1}{\longrightarrow} \bI_0(X) 
\]
is a chain complex. We let \(H_k^{\IC}(X)= \ker \partial_{k}/ \,\text{im}\, \partial_{k+1}\), for \(k\geq 1\), denote the \(k\)-th homology group of this chain complex. We remark that if \(\Haus^{k+1}(X)=0\), then \(\bI_{k+1}(X)=\{0\}\) and thus \(H_k^{\IC}(X)\) is equal to the abelian group of all \(T\in \bI_k(X)\) with \(\partial T=0\). 

Any closed oriented Riemannian \(n\)-manifold \(M\) carries a non-trivial \(n\)-cycle in a natural way. As in the smooth case, one can integrate Lipschitz differential forms \(f d\pi_1 \wedge \dotsm \wedge d\pi_n\) on \(M\). We set
\[
\bb{M}(f,\pi) = \int_Mf\det(D\pi)\,d\hspace{-0.14em}\Haus^n
\]
for all \((f, \pi)\in \mathcal{D}^n(M)\). By virtue of Stokes theorem, this defines an integral \(n\)-cycle. Since \(M\) admits a bi-Lipschitz triangulation, it is a straightforward consequence of the constancy theorem, see e.g.~\cite[Corollary 3.13]{federer-1960}, that \(H_n^{\IC}(M)\) is infinite cyclic and \(\bb{M}\) is a generator of \(H_n^{\IC}(M)\).  

An important tool which we will frequently use is the following slicing inequality. Let \(\rho \colon X \to \R\) be a Lipschitz function. Then for any \(k\in\N\) and \(T\in \bI_k(X)\), 
\[
\langle T, \rho, r \rangle= \partial( T\on \{ \rho < r\})-(\partial T)\on \{ \rho < r\}
\]
is an integral \((k-1)\)-current in \(X\) for almost every \(r\in \R\). It can be shown that \(\spt \langle T, \rho, r \rangle \subset \spt T \cap \rho^{-1}(r)\), which in general is a strict inclusion. Moreover, 
\[
\int_{\R} \norm{\langle T, \rho, r \rangle}(A) \, dr \leq (\Lip \rho) \cdot \norm{T}\big(A\big)
\]
for all Borel subsets \(A\subset X\). This inequality is sometimes referred to as the “slicing inequality".

\subsection{Minimal fillings in injective metric spaces}  
Let \(T\in \bI_{k}(X)\) be a cycle. Recall that this means \(\partial T=0\) if \(k\geq 1\) and \(T(1)=0\) if \(k=0\). We say that \(S\in \bI_{k+1}(X)\) is a \textit{minimal filling} of \(T\) if \(\partial S=T\) and \(\mass(S)\leq \mass(S')\) for every \(S'\in \bI_{k+1}(X)\) with \(\partial S'=T\). In the following we consider minimal fillings in injective metric spaces. As in Euclidean space it can be shown that minimal fillings exist and that their masses satisfy an isoperimetric inequality and a lower density bound.

\begin{theorem}\label{thm:filling-inj-spaces}
Let \(k\) be a non-negative integer. Then there is a constant \(D=D_k\geq 1\) such that the following holds. If \(Y\) is an injective metric space and \(T\in \bI_{k}(Y)\) a cycle, then there exists a minimal filling \(S\in \bI_{k+1}(Y)\) of \(T\) and any such filling satisfies
\begin{equation}\label{eq:filling-estis}
\mass(S) \leq D \mass(T)^{\frac{k+1}{k}} \quad \text{ and } \quad \norm{S}(B(y, r))\geq D^{-k}r^{k+1}
\end{equation}
for all \(y\in \spt S\) and all non-negative \(r\) with \( d(y, \spt T) \geq r\). In particular, \(\spt S\) is contained in the \(R\)-neighborhood of \(\spt T\) for \(R=D\mass(S)^{\frac{1}{k+1}}\).
\end{theorem}
If \(k=0\) then the first inequality in \eqref{eq:filling-estis} is understood as an empty statement. 

\begin{proof}
Let \(Y\) be an injective metric space and \(T\in \bI_k(Y)\) a cycle. By \cite[Theorem~1.3]{wenger-2014}, \(T\) admits a filling \(S\in \bI_{k+1}(Y)\) such that \(\mass(S)\leq \mass(S')\) for every filling \(S'\in \bI_{k+1}(Y)\) of \(T\). First, we suppose that \(k=0\). By \cite{bonicatto-del-nin-2022}, we know that
\[
S=\sum_{i \in \N} S_i \quad \text{ and } \quad \bN(S)=\sum_{i\in \N} \bN(S_i)
\]
where \(S_i=\gamma_{\#}^i\big(\bb{\mathbbm{1}_{[0,1]}}\big)\) with \(\gamma^i\colon [0,1]\to Y\) an injective Lipschitz curve or injective Lipschitz loop. Since \(S\) is a minimal filling of \(T\), it follows that no Lipschitz loops with positive mass can appear in the decomposition. Moreover, every injective Lipschitz curve must be a geodesic. Hence, \eqref{eq:filling-estis} follows. 

In the following,  we consider the case when \(k\geq 1\). A standard argument (see \cite[Theorem~10.6]{ambrosio-kirchheim-2000} or \cite[Lemma~3.4]{wenger-2005}) shows that if \(Y\) admits an isoperimetric inequality of
Euclidean type for \(\bI_k(Y)\) with constant \(C\), then any minimal filling \(S\) satisfies \(\norm{S}(B(y, r))\geq D^{-k}r^{k+1}\) for all \(y\in \spt S\) and all \(r\in \big[0,\,d(y, \spt T)\big]\), where \(D\) depends only on \(C\) and \(k\). Thus, it remains to show that there is such a \(C>0\) which is independent of \(Y\).  But this follows directly from \cite[Corollary~1.3]{wenger-2005}, since we may suppose that \(Y\subset E\) for some Banach space \(E\) and  that there is a \(1\)-Lipschitz retraction \(E\to Y\).
\end{proof}


\section{Lower bound on the Hausdorff measure}\label{sec:lower-bound-Hausdorff}

The aim of this short section is to prove the following result which generalizes \cite[Corollary 1.4]{Kinneberg-2018} to the setting of metric spaces which are not necessarily doubling and which will allow us to apply Theorem~\ref{thm:bates-thm} in the proof of Theorem~\ref{thm:main}.

\begin{theorem}\label{thm:lower-bound-Hausdorff-measure}
    Let $X$ be a metric space which has finite Hausdorff $n$-measure and is homeomorphic to a closed topological $n$-manifold. If $X$ is linearly locally contractible then there exists $c>0$ such that $$\Haus^n(B(x,r))\geq c\cdot r^n$$ for every $x\in X$ and every $0\leq r\leq \diam X$. The constant $c$ only depends on $n$ and on the linear local contractibility constant.
\end{theorem}

The overall strategy of proof is similar to that of \cite{Kinneberg-2018}. Our main new ingredient is the lemma below whose proof relies on the recent work \cite{LLN-22}. In what follows, dimension refers to the topological dimension.

\begin{lemma}\label{lem:factor-through-low-dim-simpl}
 Given a compact and linearly locally contractible metric space $X$ of dimension $n\geq 2$ there exists $c>0$ only depending on the data of $X$ such that for every $0<r\leq\diam X$ the following holds. Let $S\subset X$ be a closed subset satisfying $$\Haus^{n-1}(S)<c\cdot r^{n-1}$$ and let $\iota_S\colon S\hookrightarrow X$ be the inclusion map. Then there exist a metric space $Z$ of dimension at most $n-2$ and continuous maps $f\colon S\to Z$ and $g\colon Z\to X$ such that $\iota_S$ is homotopic to $g\circ f$ via a homotopy whose image lies in the $r$-neighborhood of $S$.
\end{lemma}

\begin{proof}
Let $c>0$ be sufficiently small, to be determined later, and let $0<r\leq \diam X$. View $X$ as a subset of the Banach space $l^\infty$ of bounded sequences equipped with the sup norm, which we denote by $\| \cdot \|$. By \cite[Theorem~3.1]{LLN-22}, see also \cite[Theorem~1.5]{avvakumov2023boxing}, there exists a finite $(n-2)$-dimensional simplicial complex $K \subset l^\infty$ and a Lipschitz map $f\colon S \to K$ such that
\begin{equation}\label{eq:homotopy-f}
        \|x - f(x)\| \leq c_1 \cdot \Haus^{n-1}(S)^\frac{1}{n-1}< c_1 c^{\frac{1}{n-1}}r
\end{equation}
for all $x \in S$, where $c_1$ is a constant only depending on $n$. In particular, the compact set $Z:=f(S)$ lies in the closed $c'r$-neighborhood of $S$, where $c'= c_1 c^{\frac{1}{n-1}}$. 
    
From now on, we assume that $c>0$ is so small that $16Q^2c'<1$, where $Q\geq 1$ is as in Proposition~\ref{prop:eps-cont-extension}. Let $g_0 \colon Z \to X$ be a (possibly discontinuous) map satisfying $\|g_0(z)-z\| \leq 2 d(z,X)$ for every $z \in Z$. Since $g_0$ is $5c'r$-continuous and $5c'<Q^{-1}$ it follows from Proposition~\ref{prop:eps-cont-extension} that there exists a continuous map $g\colon Z\to X$ with $d(g,g_0)<5Qc'r$. We have 
$$
d(x, g(f(x)))\leq \|x-f(x)\| + \|f(x)-g_0(f(x))\| + \|g(f(x))-g_0(f(x))\|
$$ 
for every $x\in S$ and hence $d(\iota_S, g\circ f)< (3+5Q)c'r$. By the choice of $c$ and the remark after Proposition~\ref{prop:eps-cont-extension} we see that $\iota_S$ and $g\circ f$ are homotopic via a homotopy with image in the $r$-neighborhood of $S$. This completes the proof.
\end{proof}

We can now provide the proof of the main result of this section.

\begin{proof}[Proof of Theorem~\ref{thm:lower-bound-Hausdorff-measure}] The proof when $n=1$ is easy and is left to the reader. We therefore assume that $n\geq 2$. Let $x\in X$ and for $r>0$ define $S_r:= \{x'\in X: d(x, x')=r\}$. Let $c>0$ be as in Lemma~\ref{lem:factor-through-low-dim-simpl}. By the coarea inequality it suffices to show that 
$$
\Haus^{n-1}(S_r)\geq 4^{-(n-1)}\cdot c \cdot r^{n-1}
$$ 
for almost every $r\in (0,8^{-1}\diam X)$. We argue by contradiction and assume that there exists $r\in(0,8^{-1}\diam X)$ such that the set $S:=S_r$ satisfies $\Haus^{n-1}(S)<4^{-(n-1)}c r^{n-1}$. By Lemma~\ref{lem:factor-through-low-dim-simpl} there exist a metric space $Z$ of topological dimension at most $n-2$ and continuous maps $f\colon S\to Z$ and $g\colon Z\to X$ such that the inclusion $S\hookrightarrow X$ is homotopic to $g\circ f$ via a homotopy whose image lies in the $\frac{r}{4}$-neighborhood of $S$. In the following, we abbreviate \(h=g\circ f\). Consider the following commutative diagram in \v{C}ech cohomology:
 
\begin{center}
\begin{tikzcd}
\check{H}^{n-1}(h(S))  \arrow{r}{h^\ast} \arrow{d}{g^\ast} &\check{H}^{n-1}(S)  \\
\check{H}^{n-1}(Z) \arrow[ur, swap, "f^\ast"] 
\end{tikzcd}
\end{center}
 
 The space $Z$ has topological dimension at most $n-2$ and hence $\check{H}^{n-1}(Z)$ vanishes, in particular, the induced homomorphism $h^*$ is the zero map. However, by fixing $y\in X$ with $d(x,y)\geq 2r$, we find that $S$ separates $x$ and $y$, and by the above the inclusion \(S\hookrightarrow X\) is homotopic to \(h\) through maps whose images do not meet \(\{x, y\}\).  Therefore, \cite[Lemma~2.1]{Kinneberg-2018} implies that $h^*$  is non-trivial, which is a contradiction.
\end{proof}

\section{Proof of Theorem~\ref{thm:main}}\label{sec:main-result}
Let $X$ be a metric space as in Theorem~\ref{thm:main} and let \(X=E\cup S\) be a partition where \(E\) is \(n\)-rectifiable and \(S\) purely \(n\)-unrectifiable. Theorem~\ref{thm:lower-bound-Hausdorff-measure} and the remarks preceding Theorem~\ref{thm:bates-thm} imply that  \(\Theta_{\ast n}(S, x)>0 \)
for \(\Haus^n\)-almost every \(x\in S\). As \(E\) is \(n\)-rectifiable there exists a countable collection of bi-Lipschitz maps \(\rho_i \colon K_i \to X\) with \(K_i\subset \R^n\) compact such that the images $\rho_i(K_i)$ are pairwise disjoint and 
\[
\Haus^n\Big(E \setminus \bigcup_{i\in \N} \rho_i(K_i)\Big)=0.
\]
We remark that we do not exclude the possibility that \(K_i=\varnothing\). In fact, \textit{a priori} it could even be that \(E=\varnothing\). The proof of Proposition~\ref{prop:mapping-T-to-M} will however show that \(\Haus^n(E)>0\). We begin with the following simple observation. 

\begin{lemma}
For each $i\in\N$ let $\theta_i\colon K_i\to\R$ be a measurable function satisfying $\abs{\theta_i}=1$ almost everywhere. Then 
\begin{equation}\label{eq:def-integral-current-main-Thm}
T=\sum_{i\in \N} \rho_{i\#} \bb{\theta_i}
\end{equation}
defines an integer rectifiable $n$-current in $X$ with $\norm{T}\leq n^{\frac{n}{2}}\cdot \Haus^n$.
\end{lemma}

\begin{proof}
Setting $T_i= \rho_{i\#} \bb{\theta_i}$ and $T^k = T_1+\dots+T_k$ and noticing that the sets $\rho_i(K_i)$ are pairwise disjoint we clearly have $$\norm{T^k} = \norm{T_1}+\dots+\norm{T_k}$$ for every $k\in\N$. Moreover, for any $(f,\pi)\in \mathcal{D}^n(X)$ and any Lipschitz extension $\psi$ of $\pi\circ\rho_i$,
\[
|T_i(f,\pi)|\leq \int_{K_i} |f\circ\rho_i|\cdot |\det(D\psi)|\,d\mathscr{L}^n = \int_{\pi(\rho_i(K_i))}\left(\int_{(\pi\circ\rho_i)^{-1}(y)}|f\circ\rho_i|\,d\Haus^0\right)\,d\mathscr{L}^n(y), 
\]
where the equality is due to the change of variables formula. Hence, using the coarea inequality, we find that
\[
|T_i(f,\pi)| \leq \Lip(\pi)^n\int_{\rho_i(K_i)}|f|\,d\hspace{-0.14em}\Haus^n,
\]
from which it follows that $\norm{T_i}\leq n^{\frac{n}{2}}\Haus\on\rho_i(K_i)$. Since $\Haus^n(X)<\infty$ we conclude that $T^k$ converges in mass and the limit $T$ satisfies $$\norm{T} = \sum_{i=1}^\infty\norm{T_i}\leq n^{\frac{n}{2}}\cdot \Haus^n.$$ This shows that $T$ is integer rectifiable and satisfies $\norm{T}\leq n^{\frac{n}{2}}\cdot \Haus^n$. 
\end{proof}

We now want to show that by choosing the $\theta_i$ appropriately, the current $T$ has zero boundary.

\begin{proposition}\label{prop:def-of-theta-i}
There exist measurable functions $\theta_i\colon K_i\to \R$ satisfying $\abs{\theta_i}=1$ almost everywhere such that $T$ as defined in \eqref{eq:def-integral-current-main-Thm} satisfies $\partial T=0$.
\end{proposition}

We will need the following lemma.

\begin{lemma}\label{lem:good-continuous-extension-LLC}
Let $Y$ be a linearly locally contractible metric space. Suppose $K\subset\R^n$ is compact and $\rho\colon K\to Y$ is bi-Lipschitz. Then there exist $C>0$ and a continuous extension $\tilde{\rho}\colon U\to Y$ of $\rho$ to an open neighborhood $U$ of $K$ such that 
\begin{equation}\label{eq:lip-property-cont-ext}
     d(\rho(z), \tilde{\rho}(w)) \leq C\, |z-w|
\end{equation}
for every $z\in K$ and every $w\in U$. Every such extension has the following properties.  For almost every $z\in K$ there exists $r>0$ such that $\rho(z)\notin\tilde{\rho}(B(z,r)\setminus\{z\})$. Moreover, if $\pi\colon Y\to \R^n$ is Lipschitz then $\pi\circ\tilde{\rho}$ is differentiable at almost every $z\in K$.
\end{lemma}

The constant $C>0$ only depends on $n$, the linear local contractibility constant, and the bi-Lipschitz constant of $\rho$.

\begin{proof}
Let \(\mathcal{F}\) be a Whitney cube decomposition of \(\R^n \setminus K\), see e.g.~\cite[Theorem VI.1.1]{Stein-singular-integrals-1970}, and let $A$ denote the union of cubes \(Q\in \mathcal{F}\) that are \(\delta\)-close to \(K\) for some \(\delta>0\) sufficiently small. For each vertex $v$ in the $0$-skeleton of $A$ choose a nearest point $a_v\in K$ and define $h(v)=\rho(a_v)$. Since $Y$ is linearly locally contractible we can continuously extend $h$ to the $k$-skeleton of $A$ for every $k=1,\dots, n$. The resulting map $h\colon A\to Y$ is continuous and satisfies \(\diam h(Q)\leq E \cdot d(Q, K)\) for every cube $Q\subset A$, where \(E>0\) is a constant depending only on \(n\), the bi-Lipschitz constant $L$ of $\rho$, and the linear local contractibility constant of $Y$. Suppose now that \(z\in K\) and \(w\in A\). Let \(Q\subset A\) be a cube that contains \(w\) and \(v\) a vertex of \(Q\). We have
\[
d(h(w), \rho(z))\leq d(h(w), h(v))+Ld(a_v, z) \leq (E+4L)
\cdot |w-z|.
\] 
Letting \(U=N_\delta(K)\), we find that \(\tilde{\rho}\colon U \to X\) defined by \(\tilde{\rho}=\rho\) on \(K\) and \(\tilde{\rho}=h\) on \(U\setminus K\) is continuous and satisfies \eqref{eq:lip-property-cont-ext}.

Now, let $\tilde{\rho}\colon U\to Y$ be any continuous extension of $\rho$ satisfying \eqref{eq:lip-property-cont-ext}. Let \(z\in K\) be a Lebesgue density point of $K$ and let $r>0$ be so small that $B(z,r)$ is contained in $U$ and such that for every \(w\in B(z, r)\) there exists \(w'\in K\) with \(\abs{w-w'} \leq (2CL)^{-1} \abs{w-z}\), where $L\geq 1$ is the bi-Lipschitz constant of $\rho$. It follows from \eqref{eq:lip-property-cont-ext} that for every such $w$ we have
$$d(\rho(z), \tilde{\rho}(w))\geq L^{-1}\abs{z-w'} - C\abs{w'-w}\geq (2L)^{-1}\abs{z-w}$$ and thus \(\tilde{\rho}(w)\neq \tilde{\rho}(z)\) for every \(w\in B(z, r)\) distinct from \(z\). 

Finally, let $\pi\colon Y\to\R^n$ be Lipschitz and set $f= \pi\circ \tilde{\rho}$. Let furthermore $\bar{f}$ be a Lipschitz extension of $\pi\circ\rho$ to $U$. We compute $$|f(w)-\bar{f}(w)|\leq (C\Lip(\pi)+\Lip(\bar{f}))\cdot d(w,K)$$ for every $w\in U$. From this it easily follows that if $z\in K$ is a Lebesgue density point of $K$ and $\bar{f}$ is differentiable at $z$ then so is $f$ and $D_zf=D_z\bar{f}$. Since $\bar{f}$ is differentiable almost everywhere on $K$ by Rademacher's theorem the proof is complete.
\end{proof}

In the proof of Proposition~\ref{prop:def-of-theta-i} we will need the notion of local degree which we now recall. Let $Y$ and $Z$ be oriented topological $n$-manifolds and let $f\colon Z\to Y$ be a continuous map. Let $z\in Z$ and suppose \(U\) is an open neighborhood of \(z\) such that \(f\) is a map of pairs \((U, U\setminus z)\to (Y, Y\setminus y)\), where \(y=f(z)\). The local degree of $f$ at $z$ is defined by
\[
\deg(f, z)=\deg_{\hspace{0.1em}\{y\}}(f|_U).
\]
This definition is independent of the open neighborhood \(U\) of \(z\). It is not difficult to show that the local degree of $f$ is the restriction of a Borel function defined on all $Z$. Notice that if $Y=Z=\R^n$ and $f$ is differentiable at $z$ with non-degenerate derivative then $$\deg(f,z) = \sgn(\det D_zf).$$

We are now in a position to provide the proof of the proposition.

\begin{proof}[Proof of Proposition~\ref{prop:def-of-theta-i}]
We begin with the following observation. Let $\rho\colon K\to X$ be a bi-Lipschitz map defined on a compact subset $K\subset\R^n$ and let $\tilde{\rho}\colon U\to X$ be a continuous extension of $\rho$ as in Lemma~\ref{lem:good-continuous-extension-LLC}. Let $\pi\colon X\to\R^n$ be Lipschitz and let $\psi$ be a Lipschitz extension of $\pi\circ\rho$. It follows from the area formula, the coarea inequality, and (the proof of) Lemma~\ref{lem:good-continuous-extension-LLC} that for almost every $y\in \R^n$ the preimage $\pi^{-1}(y)$ is finite and for every $z\in K\cap \psi^{-1}(y)$ the following holds. The maps $\psi$ and $\pi\circ\tilde{\rho}$ are differentiable with $D_z(\pi\circ\tilde{\rho}) = D_z\psi$ non-degenerate, and $\rho(z)\notin \tilde{\rho}(B(z,r)\setminus\{z\})$ for small $r>0$. In particular, the local degrees of $\tilde{\rho}$ at $z$ and of $\pi$ at $\tilde{\rho}(z)$ are defined and
\begin{equation}\label{eq:degree-pi-rhotilde}
     \deg(\tilde{\rho}, z)\cdot \deg(\pi, \tilde{\rho}(z)) = \sgn(\det D_z\psi).
\end{equation}

Now, let $\rho_i\colon K_i\to X$ be a bi-Lipschitz map as at the beginning of this section and let $\tilde{\rho}_i$ be a continuous extension of $\rho_i$ as in Lemma~\ref{lem:good-continuous-extension-LLC}. Applying the observation above to $\tilde{\rho}_i$, a Lipschitz extension $\pi$ of $\rho_i^{-1}$, and the identity map $\psi$ we obtain that $\deg(\tilde{\rho}_i,z)$ exists at almost every $z\in K_i$ and has value $-1$ or $1$. We define $\theta_i(z) = \deg(\tilde{\rho}_i,z)$; by the remark preceding the proof this defines a measurable function on $K_i$.

We finally show that with this choice of $\theta_i$ the current $T$ defined in \eqref{eq:def-integral-current-main-Thm} satisfies $\partial T=0$. For this, let $\pi\colon X\to\R^n$ be a Lipschitz map. We will show that $T(1,\pi)=0$. By Theorem~\ref{thm:bates-thm} and the continuity property for currents we may assume that $\pi$ is such that $\Haus^n(\pi(S))=0$, where \(S\subset X\) denotes the purely unrectifiable part of \(X\). Fix $i$ and let $\psi$ be a Lipschitz extension of $\pi\circ\rho_i$. The change of variables formula together with \eqref{eq:degree-pi-rhotilde} yield 
\begin{equation*}
  \begin{split}
      \rho_{i\#}\bb{\theta_i}(1,\pi) &= \int_{K_i}\theta_i(z)\det (D_z\psi)\,d\mathscr{L}^n(z)\\
      &=\int_{\R^n}\left(\sum_{x\in \pi^{-1}(y)}\mathbbm{1}_{\rho_i(K_i)}(x)\deg(\pi,x)\right)\,d\mathscr{L}^n(y).
  \end{split}
\end{equation*}
Since for every $y\in \R^n$ for which $\pi^{-1}(y)$ is finite we have $$\sum_{x\in\pi^{-1}(y)}\deg(\pi, x) = \deg(\pi) =0,$$ see e.g.~\cite[p. 267]{dold-1980}, and since $\Haus^n(\pi(S))=0$ it follows that $$T(1,\pi) = \sum_{i\in\N}\rho_{i\#}\bb{\theta_i}(1,\pi) = \int_{\R^n}\left(\sum_{x\in\pi^{-1}(y)}\deg(\pi,x)\right)\,d\mathscr{L}^n(y) = 0.$$ This concludes the proof.
\end{proof}

From now on we assume that $\theta_i$ is as in Proposition~\ref{prop:def-of-theta-i} and $T$ is defined by \eqref{eq:def-integral-current-main-Thm}. 

\begin{proposition}\label{prop:mapping-T-to-M}
For every Lipschitz map $\varphi\colon X\to M$ we have $\varphi_\#T = \deg(\varphi)\cdot\bb{M}$.
\end{proposition}

\begin{proof}
Suppose first that $\varphi\colon X\to M$ is a Lipschitz map which furthermore satisfies $\Haus^n(\varphi(S))=0$, where $S\subset X$ is the purely $n$-unrectifiable part of $X$. Let $(g,\tau)\in\mathcal{D}^n(M)$. One calculates exactly as in the proof of Proposition~\ref{prop:def-of-theta-i} that 
$$
\varphi_\#\rho_{i\#}\bb{\theta_i}(g,\tau) = \int_{\R^n}\left(\sum_{x\in \varphi^{-1}(\tau^{-1}(y))}\mathbbm{1}_{\rho_i(K_i)}(x)g(\varphi(x))\deg(\tau\circ \varphi, x)\right)\,d\mathscr{L}^n(y).
$$
For almost every $y\in\R^n$ the preimage $\tau^{-1}(y)$ is finite, does not intersect $\varphi(S)$, and for every $z\in \tau^{-1}(y)$ we have that $\varphi^{-1}(z)$ is finite and $\tau$ is differentiable at $z$ with non-degenerate derivative. For every such $y$ and every $z\in\tau^{-1}(y)$ and $x\in \varphi^{-1}(z)$ we thus have 
$$
\deg(\tau\circ \varphi, x)= \deg(\tau, z)\cdot \deg(\varphi, x) = \sgn(\det D_z\tau)\cdot \deg(\varphi, x).
$$ 
Summing over $i$ and using the fact that $\varphi^{-1}(\tau^{-1}(y))\cap S=\emptyset$ for almost every $y\in\R^n$, we conclude that 
\begin{equation*}
\begin{split}
     \varphi_\#T(g, \tau) &=\int_{\R^n}\left(\sum_{z\in\tau^{-1}(y)}\left(\sum_{x\in \varphi^{-1}(z)} \deg(\varphi, x)\right)\cdot g(z)\sgn(\det D_z\tau)\right)\,d\mathscr{L}^n(y).
\end{split}
\end{equation*}
By the use of the change of variables formula, this shows that $\varphi_\#T= \deg(\varphi)\cdot \bb{M}$ and establishes the proposition under the extra assumption that $\Haus^n(\varphi(S))=0$. 

Now, let $\varphi\colon X\to M$ be any Lipschitz map. Since $M$ admits a bi-Lipschitz embedding into some $\R^m$ and its image in $\R^m$ is a Lipschitz neighborhood retract it follows from Theorem~\ref{thm:bates-thm} that there exists a sequence $(\varphi_k)$ of Lipschitz maps $\varphi_k\colon X\to M$ converging uniformly to $\varphi$, with uniformly bounded Lipschitz constant, and such that $\Haus^n(\varphi_k(S))=0$ for every $k$. Since $\deg(\varphi_k)=\deg(\varphi)$ for all sufficiently large $k$ it follows from the definition of a current and the first part of the proof that 
\begin{align*}
\varphi_\#T(g,\tau) &= T(g\circ \varphi, \tau\circ \varphi) = \lim_{k\to \infty} T(g\circ \varphi_k, \tau\circ \varphi_k) \\
&= \lim_{k\to\infty} \varphi_{k\#}T(g,\tau) = \deg(\varphi)\cdot \bb{M}(g,\tau)
\end{align*}
for all $(g,\tau)\in\mathcal{D}^n(M)$ and thus $\varphi_{\#}T = \deg(\varphi) \cdot\bb{M}$. This concludes the proof.
\end{proof}

The following proposition shows that $T$ is a generator of the $n$-th homology group via integral currents in $X$.

\begin{proposition}\label{prop:T-is-generator}
Every cycle $T'\in\bI_n(X)$ is of the form $T'=mT$ for some $m\in\Z$.
\end{proposition}

We will need the following lemma.

\begin{lemma}\label{lem:fillrad-0-implies-0}
Let $Y$ be a complete separable metric space and suppose  $C\subset Y$ is a compact set with $\Haus^{n+1}(C)=0$ for some $n\in\N$. Let furthermore  $S\in\bI_n(Y)$ be a cycle supported in $C$. If for every $\varepsilon>0$ there exists $R\in\bI_{n+1}(Y)$ supported in $N_\varepsilon(C)$ and satisfying $\partial R = S$ then $S=0$.
\end{lemma}

\begin{proof}
We may view $Y$ as a subset of $\ell^\infty$. Since $\ell^\infty$ has the metric approximation property (see e.g.~\cite[Proposition A.6]{de-pauw-2014} for an elementary proof) there exist for every $\varepsilon>0$ a finite-dimensional subspace $V\subset\ell^\infty$ and a linear $1$-Lipschitz map $\psi\colon \ell^\infty\to V$ such that $\|y-\psi(y)\|\leq \varepsilon$ for every $y\in C$. The set $C'=\psi(C)$ is compact and satisfies $\Haus^{n+1}(C')=0$. Moreover, the current $S'=\psi_\#S\in\bI_n(V)$ is supported in $C'$ and admits a filling in the $\varepsilon'$-neighborhood of $C'$ for every $\varepsilon'>0$. Since $V$ is finite-dimensional it follows from (the proof of) Federer--Fleming's deformation theorem that $S'=0$. Indeed, let $\delta>0$ and choose a cubical subdivision of $V$ into Euclidean cubes of sidelength $\delta$. By the proof of the deformation theorem there exists a decomposition $S'= P+\partial Q$, where $P$ is a polyhedral cycle in the $n$-skeleton of the cubical subdivision and $Q\in\bI_{n+1}(V)$. Moreover, $\mass(P)\leq C\mass(S')$ and $\mass(Q)\leq C\delta \mass(S')$ for some constant \(C\) depending on the dimension $N$ of $V$. Roughly speaking, the current $P$ is obtained by pushing $S'$ via radial projections first to the $(N-1)$-skeleton, then to the $(N-2)$-skeleton, etc. until the $n$-skeleton. Since $C'$ is compact and satisfies $\Haus^{n+1}(C')=0$ we can choose the projection centers not to lie on (the projections of) $C'$. Hence, $P$ is the pushforward under a Lipschitz map $\pi$ and thus admits a filling in an arbitrary small neigbhorhood of $\pi(C')$. Since the $n$-skeleton is a Lipschitz neigborhood retract it follows that $P$ admits a filling in the $n$-skeleton. From this it follows that $P=0$ and hence $S'=\partial Q$. Since $\mass(Q)\leq C\delta \mass(S')$ and $\delta>0$ was arbitrary it follows that $S'=0$, as claimed.

Finally, let $\bb{0,1}\times S$ denote the product current (see \cite[Section 3.3]{basso2021undistorted} for the definition), and note that this is an element of $\bI_{n+1}([0,1]\times \ell^\infty)$. Let furthermore $H\colon [0,1]\times C\to \ell^\infty$ be the straight-line homotopy from the identity on $C$ to $\psi$. It follows from the homotopy formula, see e.g.~\cite[Theorem~2.9]{wenger-2005}, that $R=H_\#(\bb{0,1}\times S)\in\bI_{n+1}(\ell^\infty)$ satisfies $\partial R = S$. Moreover, \cite[Lemma 3.5]{basso2021undistorted} shows that $$\mass(R) \leq (n+1)\varepsilon \mass(S')\leq (n+1)\varepsilon \mass(S).$$
 Since $\varepsilon>0$ was arbitrary it follows that $S=0$, thus concluding the proof. 
\end{proof}

We can now provide the proof of the proposition.

\begin{proof}[Proof of Proposition~\ref{prop:T-is-generator}]
Let $T'\in \bI_n(X)$ and let $\rho\colon X\to M$ be a homeomorphism of degree one. By Lemma~\ref{lem:lip-approx-easy} there exists a Lipschitz map $f\colon X\to M$ which is homotopic to $\rho$. For any such $f$ we have  $f_\#T = \bb{M}$ and there exists $m\in\Z$ such that $f_\#T' = m\cdot \bb{M}$. Notice that $m$ is independent of the choice of $f$. We will show that $T'=mT$. For this it suffices by Lemma~\ref{lem:fillrad-0-implies-0} to prove that the cycle $S:= mT-T'$ admits a filling in the $\varepsilon$-neighborhood of $X$ in $E(X)$ for every $\varepsilon>0$.
 
By Lemma~\ref{lem:lip-approx-easy} there exist Lipschitz approximations $f\colon X\to M$ and $\eta\colon M\to E(X)$ of $\rho$ and $\rho^{-1}$, respectively, such that the Lipschitz map $\psi= \eta\circ f$ satisfies $d(\psi(x), x)\leq \varepsilon$ for every $x\in X$. Clearly, we have $f_\#S=0$ and hence $\psi_\#S=0$. Furthermore, since $E(X)$ is injective there exists a Lipschitz homotopy $H\colon [0,1]\times X\to E(X)$ from $\psi$ to the identity on $X$ with image in the closed $\varepsilon$-neighborhood of $X$. It follows that the integral $(n+1)$-current $R=H_\#(\bb{0,1}\times S)$ is a filling of $S$ supported in $N_\varepsilon(X)$. Here, as in the proof of Lemma~\ref{lem:fillrad-0-implies-0}, we use $\bb{0,1}\times S$ to denote the product current and we have used the homotopy formula to show that $\partial R=S$. Since $\varepsilon>0$ was arbitrary it follows from Lemma~\ref{lem:fillrad-0-implies-0} that $S=0$ and hence $T'=mT$, as claimed.
\end{proof}

The next proposition completes the proof of Theorem~\ref{thm:main}.

\begin{proposition}\label{prop:lower-bound-mass}
 There exists $c>0$ depending on $n$ and the linear local contractibility constant of $X$ such that 
 \begin{equation*}
\norm{T}(B(x, r)) \geq c \cdot r^n
\end{equation*}
for every \(x\in X\) and every \(r\in (0, \diam X)\).
\end{proposition}

\begin{proof}
Fix $x\in X$. Since for all $r>0$, by the slicing inequality,
$$
\int_0^r\mass(\partial(T\on B(x,s)))\,ds \leq \norm{T}(B(x,r)),
$$ 
it is enough to prove that for almost every $r\in \big(0, 2^{-3} \cdot \diam X\big)$,  
 \begin{equation}\label{eq:lower-bound-slices}
 \mass(\partial(T\on B(x,r)))\geq cn\cdot r^{n-1},
 \end{equation}
 where $c>0$ is such that $4\cdot D^{1+\frac{1}{n}}\cdot (c n)^{\frac{1}{n-1}}= (5Q +3)^{-1}$.  Here, \(D\) denotes the constant \(D_k\) from Theorem~\ref{thm:filling-inj-spaces} for \(k=n-1\) and $Q\geq 1$ is the constant from Proposition~\ref{prop:eps-cont-extension}. 
 
 Let $r\in \big(0, 2^{-3} \cdot \diam X\big)$ be such that $T\on B(x,r)\in\bI_n(X)$,  with boundary supported in $\{x'\in X: d(x, x')=r\}$, and suppose \eqref{eq:lower-bound-slices} is not true. Let $U:= \partial(T\on B(x,r))$ and let $V\in \bI_n(E(X))$ be a minimal filling of $U$. It then follows that $\spt V$ has finite Hausdorff $n$-measure and thus topological dimension at most $n$; moreover, $\mass(V)\leq c' \cdot r^n$ and
 \begin{equation}\label{eq:control-support-V}
      d_H(\spt V, \spt U)\leq c'\cdot r,
 \end{equation}
where \(d_H\) denotes the Hausdorff distance and \(c'= D^{1+\frac{1}{n}}\cdot (c n)^{\frac{1}{n-1}}\). 

Set $Y:= X\cup \spt V$. We construct a continuous map $\pi\colon Y\to X$ which restricts to the identity on $X$ and such that 
\begin{equation}\label{eq:property-retraction-eps-cont}
    d_H(\pi(\spt V), \spt U)\leq \frac{r}{4}.
\end{equation} 
For this, let $f\colon Y\to X$ be a (possibly discontinuous) map satisfying $d(f(y), y)\leq 2d(y, X)$ for every $y\in Y$. Notice that $f$ restricts to the identity on $X$ and 
$$
d(f(y), f(y'))\leq 2 d(y, X) + 2 d(y', X) + d(y,y')
$$ 
for all $y, y'\in Y$. This together with \eqref{eq:control-support-V} implies that $f$ is $5c'r$-continuous and moreover continuous at every point of $X$. Since $5c'r<Q^{-1}\diam X$ there exists by Proposition~\ref{prop:eps-cont-extension} a continuous map $\pi\colon Y\to X$ which agrees with $f$ on $X$ and satisfies $d(\pi, f)<5Qc'r$. It is easy to show that $\pi$ satisfies \eqref{eq:property-retraction-eps-cont}. Indeed, let $y\in\spt V$ and notice that by \eqref{eq:control-support-V} there exists $z\in\spt U$ with $d(y,z)\leq c'r$. We estimate
\begin{equation*}
        d(\pi(y), z)\leq d(\pi(y), f(y)) + d(f(y), y) + d(y,z)< (5Q+3)c'r\leq \frac{r}{4},
\end{equation*}
from which \eqref{eq:property-retraction-eps-cont} follows.

Set $h= \frac{r}{4}$ and let $\varrho\colon X\to M$ be a homeomorphism of degree one. Let $\varepsilon>0$ be as in Lemma~\ref{lem:technical-inclusions-approx-homeo} below. As $X$ is compact and $\varrho$ a homeomorphism, there exists $\varepsilon'>0$ such that $$N_{\varepsilon'}(\varrho(A))\subset \varrho(N_h(A))$$ for all $A\subset X$.
Let $\varphi\colon Y\to M$ be a Lipschitz map satisfying $d(\varphi, \varrho\circ \pi)<\min\{\varepsilon, \varepsilon'\}$. Due to \eqref{eq:property-retraction-eps-cont} we have
$$
\spt (\varphi_\#V) \subset N_{\varepsilon'}\big(\varrho(\pi(\spt V))\big)\subset \varrho\big(N_h(\pi(\spt V))\big)\subset \varrho\big(\big\{x'\in X: \frac{r}{2}<d(x',x)<\frac{3r}{2}\big\}\big).
$$
It follows that the integral $n$-cycle in $M$ given by $W:=\varphi_\#(T\on B(x,r) - V)$ has support in $\varrho(B(x, \frac{3r}{2}))\not= M$ and hence $W=0$. Since 
$$
\varphi_\#(T\on B(x,r)) = \bb{M} - \varphi_\#(T\on (X\setminus B(x,r)))
$$ 
and, by Lemma~\ref{lem:technical-inclusions-approx-homeo},
$$
\spt\big(\varphi_\#(T\on (X\setminus B(x,r))\big) \subset \varphi(X\setminus B(x,r))\subset \varrho(X\setminus B(x,r/2)),
$$ 
we find that 
$$
\varrho(B(x,r/2))\subset \spt\big(\varphi_\#(T\on B(x,r))\big).
$$ 
It follows from this that $W\not=0$, which is impossible. This contradiction concludes the proof.
\end{proof}

In the proof above we used the following technical lemma which will also be used in the next proposition as well as in Section~\ref{sec:relative-isoperimetric-inequality}.

\begin{lemma}\label{lem:technical-inclusions-approx-homeo}
 Let \(\varrho\colon X \to M\) be a homeomorphism between compact metric spaces. Then for every $h>0$ there exists $\varepsilon>0$ with the following property. If $\varphi\colon X\to M$ is continuous with $d(\varphi,\varrho)<\varepsilon$ then $$\varphi(N_s(A))\subset\varrho(N_{s+h}(A)) \, \,  \text{ and } \, \, \varphi(X\setminus N_s(A))\subset \varrho(X\setminus N_{s-h}(A))$$ for all $s>0$ and $A\subset X$.
\end{lemma}

\begin{proof}
Let $h>0$. Since $\varrho$ is a homeomorphism there exists $\epsilon>0$ such that for all $x,y \in X$ with $d(\varrho(x),\varrho(y)) \leq \epsilon$ we have $d(x,y) \leq h$. Now, let $A \subset X$ and $s>0$. If $x \in N_s(A)$, then there exists $y \in X$ with $\varphi(x) = \varrho(y)$. Thus $d(\varrho(x),\varrho(y)) = d(\varrho(x),\varphi(x)) \leq \epsilon$. It follows that $d(x,y) \leq h$ and hence $y \in N_{s+h}(A)$. In particular, $\varphi(x) \in \varrho(N_{s+h}(A))$. 

On the other hand, if $x \in X \setminus N_s(A)$, then by the same reasoning as before there is $y \in X$ with $\varphi(x) = \varrho(y)$ and $d(x,y) \leq h$. Therefore, $x \in X \setminus N_{s-h}(A)$ and $\varphi(x) \in \varrho(X \setminus N_{s-h}(A))$.
\end{proof}

This completes the proof of Theorem~\ref{thm:main}. We now provide the very short proof of our rectifiability result.

\begin{proof}[Proof of Corollary~\ref{cor:recitifiability}]
By Theorem~\ref{thm:main} the $n$-th homology group $H_n^{\IC}(X)$ of $X$ via integral currents is infinite cyclic and every generator $T$ satisfies $\norm{T}(B(x,r))\geq cr^n$ for all $x\in X$ and all $r\in(0,\diam X)$, where $c>0$ only depends on the data of $X$. In particular it follows that $\set T = X$, where $\set T$ denotes the characteristic set of $T$ (see \eqref{eq:characteristic-set} for the definition). By \cite[Theorem~4.6]{ambrosio-kirchheim-2000}, it follows that $\set T$ is $n$-rectifiable and thus the proof is complete. 
\end{proof}

We remark that it has recently been shown in \cite{david-american-2016} that if $X$ is as in Corollary~\ref{cor:recitifiability} and moreover Ahlfors $n$-regular then $X$ is even uniformly rectifiable. The arguments in \cite{david-american-2016} heavily rely on results of Semmes \cite{semmes-1996}. 

As alluded to in the introduction, one cannot expect a lower bound on the mass measure of a cycle as in Theorem~\ref{thm:main} without the assumption of linear local contractibility. Nevertheless, the following simple proposition shows that even without the linear local contractibility condition the support of a cycle as in the theorem is the entire space $X$.

\begin{proposition}\label{prop:support-is-all-of-X}
Let $X$ be a metric spaces homeomorphic to a closed, oriented, smooth $n$-manifold $M$. If $T\in\bI_n(X)$ is a cycle satisfying $\varphi_\#T = \deg(\varphi)\cdot\bb{M}$ for every Lipschitz map $\varphi\colon X\to M$ then $\spt T = X$.
\end{proposition}

\begin{proof}
Suppose by contradiction that $\spt T\not= X$ and let $x\in X$ and $r>0$ be such that $B(x, 2r)$ does not intersect $\spt T$. Let $\varrho\colon X\to M$ be a positively oriented homeomorphism and let $\varepsilon>0$ be as in Lemma~\ref{lem:technical-inclusions-approx-homeo} for $h=r$. By Lemma~\ref{lem:lip-approx-easy} there exists a Lipschitz map $\varphi\colon X\to M$ with $d(\varphi, \varrho)<\varepsilon$. We may assume that $\varepsilon>0$ is so small that $\varphi$ is homotopic to $\varrho$ and hence $\deg(\varphi)=1$. Let $U$ be the non-empty open subset of $M$ given by $U=\varrho(B(x, r))$ and notice that by Lemma~\ref{lem:technical-inclusions-approx-homeo} we have that $\varphi^{-1}(U)\subset B(x, 2r)$ and hence $\varphi^{-1}(U)\cap \spt T=\varnothing$. Since $\varphi_\#T=\bb{M}$ it follows that $$\bb{M}\on U = (\varphi_\#T)\on U = \varphi_\#(T\on \varphi^{-1}(U)) = 0,$$ which is a contradiction. This shows that $\spt T=X$.
\end{proof}


\section{Proof of Theorem~\ref{thm:existence-intcurr-dimension2-intro}}\label{sec:2-dim-existence}

In this section we prove Theorem~\ref{thm:existence-intcurr-dimension2-intro} from the introduction. We mention that a version of this theorem for quasiconvex metric surfaces can easily be proved with elementary methods, see the proof of Theorem~\ref{thm:Lipschitz-Volume-rigidity-dim2-quasiconvex} below. The proof of the general case (without assuming quasiconvexity) crucially relies on the recent uniformization result \cite{Ntalampekos-Romney-22}. It moreover uses the following pushforward construction for currents under Sobolev maps. Let \(M\) be a closed, oriented, Riemannian \(n\)-manifold, \(X\) a complete  metric space and \(\rho\in W^{1, n}(M, X)\) a continuous Sobolev map. For every \((f, \pi)\in \mathcal{D}^n(X)\) we define 
\begin{equation}\label{eq:def-pushforward-Sobolev}
T(f,\pi)\coloneqq \int_M f\circ\rho \,\det\big(D(\pi\circ\rho)\big)\,d\hspace{-0.14em}\Haus^n,
\end{equation}
where $D(\pi\circ\rho)$ denotes the weak derivative. The next proposition shows that this formula defines an integral current in $X$. Notice that \eqref{eq:def-pushforward-Sobolev} is exactly the definition of the pushforward of $\bb{M}$ under $\rho$ in the case that $\rho$ is Lipschitz. 

\begin{proposition}\label{prop:Sobolev-pushforward-currents}
The functional $T$ defines an element of $\bI_n(X)$ with $\partial T=0$ and $$\norm{T}(A) \leq \vol^*(\rho|_{\rho^{-1}(A)})$$ for every Borel set $A\subset X$. Moreover, if $X$ is homeomorphic to $M$ and $\rho$ is a uniform limit of positively oriented homeomorphisms from $M$ to $X$ then $$\varphi_\#T = \deg(\varphi)\cdot\bb{M}$$ for every Lipschitz map $\varphi\colon X\to M$.
\end{proposition}

For more general results about pushing forward integral currents under Sobolev maps, proved with different methods, see the recent article \cite{Ikonen-2023}.

\begin{proof}
Since $(M, d_g)$ has a $1$-Poincar\'e inequality it follows from \cite[Theorem 8.1.7]{Heinonen-Koskela-Shanmugalingam-Tyson-2015} that there exist $h\in L^p(M)$ and $B\subset M$ negligible such that $$d(\rho(z), \rho(w))\leq d_g(z, w)\cdot (h(z) + h(w))$$ for all $z,w\in M\setminus B$. For $k\in\N$ define $A_k= \{z\in M\setminus B: h(z)\leq k\}$ and $$\varepsilon_k:= \int_{M\setminus A_k} h^n\,d\Haus^n.$$ It follows from Chebyshev's inequality and the absolute continuity of the integral that $$\Haus^n(M\setminus A_k) \leq \frac{\varepsilon_k}{k^n}$$ and that $\varepsilon_k\to 0$ as $k\to \infty$.
Since the restriction of $\rho$ to $A_k$ is $(2k)$-Lipschitz there exists a $(2k)$-Lipschitz extension $\rho_k\colon M\to E(X)$ of $\rho|_{A_k}$. It is not difficult to show that $(\rho_k)$ converges uniformly to $\rho$ and, in particular, $\rho_k(M)\subset N_{\delta_k}(X)$ for some $\delta_k\to 0$.

Now, define $T_k:= \rho_{k\, \#}\bb{M}$ and notice that $T_k\in\bI_n(E(X))$ satisfies $\partial T_k=0$ and $\spt(T_k)\subset \bar{N}_{\delta_k}(X)$. We set furthermore $R_k:= \rho_{k\,\#}\bb{M\setminus A_k}$ and notice that $$T_k = \rho_{k\,\#}\bb{A_k} + R_k$$ as well as $$\mass(R_k)\leq \Lip(\rho_k)^n\cdot \Haus^n(M\setminus A_k) \leq 2^n\varepsilon_k.$$ In particular, for any Borel set $E\subset E(X)$ we have 
\begin{equation}\label{eq:mass-norm-T_k-Sobolev}
\norm{T_k}(E)\leq \norm{\rho_{k\,\#}\bb{A_k}}(E) + \norm{R_k}(E)\leq \vol^*(\rho|_{A_k\cap \rho^{-1}(E)}) + 2^n\varepsilon_k.
\end{equation}
Since by the absolute continuity of the integral we have $\vol^*(\rho|_{M\setminus A_k})\to 0$ it easily follows that $$T_k(f,\pi)\to T(f,\pi)$$ for every $(f,\pi)\in \mathcal{D}^n(X)$. Moreover, for $k, m\in\N$ with $k>m$ we have $$T_k-T_m = \rho_{k\,\#}\bb{A_k\setminus A_m} + R_k - R_m$$ and hence $$\mass(T_k- T_m) \leq \vol^*(\rho|_{A_k\setminus A_m}) + \mass(R_k) + \mass(R_m)\to 0$$ as $k, m\to \infty$. This shows that $(T_k)$ is a Cauchy sequence with respect to the mass norm. Since $T_k$ converges weakly to $T$ it follows that $T\in\bI_n(X)$ and $\partial T=0$. Furthermore, we deduce from \eqref{eq:mass-norm-T_k-Sobolev} and the convergence in mass that $$\norm{T}(A)\leq \vol^*(\rho|_{\rho^{-1}(A)})$$ for every Borel set $A\subset X$. 
 
Finally, suppose that $X$ is homeomorphic to $M$ and $\rho$ is in addition the uniform limit of positively oriented homeomorphisms from $M$ to $X$. Let $\varphi\colon X\to M$ be a Lipschitz map. Let $\bar{\varphi}\colon N_\delta(X)\to M$ be a Lipschitz extension of $\varphi$ to a small neighborhood $N_\delta(X)$ of $X$ in $E(X)$. Then, $\bar{\varphi}\circ\rho_k$ converges uniformly to $\varphi\circ\rho$ and hence $\deg(\bar{\varphi}\circ\rho_k) = \deg(\varphi)$ for all sufficiently large $k$. Since $\bar{\varphi}_\#T_k$ converges in mass to $\varphi_\#T$ and $(\bar{\varphi}\circ\rho_k)_\#\bb{M} = \deg(\bar{\varphi}\circ\rho_k)\cdot \bb{M}$ for every $k$ it follows that $\varphi_\#T = \deg(\varphi)\cdot\bb{M}$.
\end{proof}

Theorem~\ref{thm:existence-intcurr-dimension2-intro} is an almost immediate consequence of Proposition~\ref{prop:Sobolev-pushforward-currents} together with the recent uniformization result \cite{Ntalampekos-Romney-22}, see also \cite{Meier-Wenger, Ntalampekos-Romney-21} for earlier uniformization results applying to the special case of quasiconvex metric surfaces of finite Hausdorff measure.

\begin{proof}[Proof of Theorem~\ref{thm:existence-intcurr-dimension2-intro}]
By \cite[Theorem 1.3]{Ntalampekos-Romney-22} there exists a Riemannian metric $g$ on $M$ of constant curvature and a continuous, surjective, monotone map $\rho\colon M\to X$ which belongs to $W^{1,2}(M, X)$ and is weakly conformal in the sense that there exists $K>0$ such that the minimal weak upper gradient $g_\rho$ of $\rho$ satisfies $$g_\rho^2\leq K\cdot \mathbf{J}(\apmd\rho)$$ almost everywhere. We refer to \cite{Heinonen-Koskela-Shanmugalingam-Tyson-2015} for the definition of minimal weak upper gradient. It follows in particular that $\rho$ is the uniform limit of homeomorphisms from $M$ to $X$. We may assume these to be positively oriented.

By Proposition~\ref{prop:Sobolev-pushforward-currents} the functional $T\colon \mathcal{D}^2(X)\to\R$ defined by 
$$
T(f,\pi)\coloneqq \int_M f\circ\rho \,\det\big(D(\pi\circ\rho)\big)\,d\hspace{-0.14em}\Haus^n
$$ 
belongs to $\bI_2(X)$ and satisfies $\partial T=0$, as well as 
$$
\|T\|(A)\leq \vol^*(\rho|_{\rho^{-1}(A)})
$$ 
for every Borel set \(A\subset X\). By \cite[Lemma 7.8]{Ntalampekos-Romney-21} we have that $\#\{z\in M: \rho(z)=x\}=1$ for $\Haus^2$-almost every $x\in X$. Recall from the proof of Proposition~\ref{prop:Sobolev-pushforward-currents} that there exist measurable sets $A_1\subset A_2\subset \dots \subset M$ such that $\rho|_{A_k}$ is Lipschitz for every $k$ and $\bigcup A_k$ has full measure in $M$. The area formula for Lipschitz maps thus implies that for every Borel set $A\subset X$ we have $$\vol(\rho|_{\rho^{-1}(A)}) =\lim_{k\to\infty} \int_A\#\{z\in A_k: \rho(z) = x\}\,d\Haus^2(x) \leq \Haus^2(A).$$ We conclude that $$\norm{T}(A) \leq \vol^*(\rho|_{\rho^{-1}(A)})\leq C \vol(\rho|_{\rho^{-1}(A)})\leq C \Haus^2(A)$$ for every Borel set $A\subset X$, where $C$ is a universal constant.
 
Finally, Proposition~\ref{prop:Sobolev-pushforward-currents} shows that $\varphi_\#T = \deg(\varphi)\cdot \bb{M}$ for every Lipschitz map $\varphi\colon X\to M$. It now follows exactly as in the proof of Proposition~\ref{prop:T-is-generator} that $H^{\IC}_2(X)$ is infinite cyclic.
\end{proof}


\section{Relative isoperimetric and Poincar\'e inequalities}\label{sec:relative-isoperimetric-inequality} The purpose of this section is to prove the relative isoperimetric inequality stated in Theorem~\ref{thm:relative-isoperimetric-inequality}. At the end of the section, we also deduce Semmes' theorem about the validity of a Poincar\'e inequality; see Corollary~\ref{cor:Semmes-theorem-Poincare}. In preparation of the proof of Theorem~\ref{thm:relative-isoperimetric-inequality} we prove the following proposition.

\begin{proposition}\label{prop:good-retraction-property}
Let $Y$ be a complete metric space and $X\subset Y$ a non-empty closed subset of finite Nagata dimension. If $X$ is linearly locally contractible, then there exist $\lambda>0$, $C\geq 1$ and a continuous retraction $\pi\colon N_{\lambda\diam(X)}(X)\to X$ such that $$d(\pi(y), x)\leq C\cdot d(y,x)$$ for all $x\in X$ and $y\in N_{\lambda\diam(X)}(X)$. The constants $\lambda, C$ only depend on the data of $X$.
\end{proposition}

The proof is similar to the first part of the proof of Lemma~\ref{lem:good-continuous-extension-LLC}. 

\begin{proof}
We denote by $k\geq 0$ the Nagata dimension of $X$. By the proof of \cite[Theorems $1.6$ and $5.2$]{lang-2005} there exists a $(k+1)$-dimensional simplicial complex $\Sigma$, a continuous map $g \colon Y \setminus X \to \Sigma$ with $\Hull (g(Y \setminus X)) = \Sigma$ and a map $h\colon \Sigma^{(0)} \to X$, defined on the $0$-skeleton of $\Sigma$, such that the following holds. For every simplex $\sigma \subset \Sigma$ and every vertex $v \in \sigma$ we have $$d(y,h(v)) \leq C d(y,X)$$
for all $y \in g^{-1}(\st \sigma)$, where $C>0$ only depends on the data of $X$. The relevant notation concerning simplicial complexes can be found in \cite[Section 2.2]{basso2021undistorted}. We denote by $\Lambda\geq 1$ the linear local contractibility constant of $X$ and set $\lambda \coloneqq [3 C (2\Lambda)^{k+2}]^{-1}$.
We replace $Y$ by $\bar{N}_{\lambda \diam X}(X)$ and $\Sigma$ by $\Hull (g(Y \setminus X))$. Then, for every simplex $\sigma \subset \Sigma$ and each $y \in Y \setminus X$ with $g(y) \in \st \sigma$, 
$$        
\diam h(\sigma^{(0)}) \leq 2Cd(y,X) \leq (2\Lambda)^{-(k+2)}\diam X.
$$
Thus, we can inductively extend $h$ to $\Sigma^{(\ell)}$ for $\ell=1,\ldots,k+1$ to obtain a map $h\colon \Sigma \to X$ with
\(\diam h(\sigma) \leq C (2\Lambda)^{k+2} d(y,X)\)
for every simplex $ \sigma \subset \Sigma$ and each $y \in Y \setminus X$ with $g(y) \in \st \sigma$. Now, let $y \in Y \setminus X$ and $\sigma \subset \Sigma$ be the minimal simplex with $g(y) \in \interior \sigma$. Further, let $v \in \sigma$ be some vertex and $x \in X$. Then 
\[
d\big(h(g(y)),x\big)\leq d\big(h(g(y)),h(v)\big) + d(h(v),y) + d(y,x)
\]
and we conclude
\[
d\big(h(g(y)),x\big)\leq \big(C(2\Lambda)^{k+2} + C +1 \big)\cdot d(y,x).
\]
We define the retraction $\pi \colon Y \to X$ as $h \circ g$ on $Y\setminus X$ and $\id_X$ on $X$.
\end{proof}

We now turn to the proof of the relative isoperimetric inequality and let \(X\) be as in Theorem~\ref{thm:relative-isoperimetric-inequality}. Since \(X\) is Ahlfors \(n\)-regular, and thus doubling, it follows that \(X\) satisfies \(\text{Nagata}(N,2)\) for some \(N\in \N\) depending only on the doubling constant of \(X\). We remark that even \(\dim_N(X)=n\), but we will not need this below. We proceed by fixing some notation.  By applying  Proposition~\ref{prop:good-retraction-property} with \(Y=E(X)\), we find that there exist constants \(\lambda>0\), \(C_1\geq 1\) depending only on the data of \(X\) such that there is a continuous retraction
\[
\pi \colon \bar{N}_{\lambda \diam(X)}^{E(X)}(X)\to X
\]
satisfying \(d(\pi(y), x) \leq C_1 \cdot d(y, x)\) for all \(x\in X\) and all \(y\in \bar{N}_{\lambda \diam(X)}^{E(X)}(X)\). Let \(\rho \colon X\to M\) be a homeomorphism of degree one to a closed, oriented, Riemannian \(n\)-manifold \(M\). By Theorem~\ref{thm:main}, there exists \(C_2\geq 1\) depending only on \(n\) and a cycle \(T\in \bI_n(X)\) with \(\norm{T} \leq C_2 \cdot\Haus^n\) such that if \(\varphi\colon X \to M\) is a Lipschitz map of degree one, then \(\varphi_\# T=\bb{M}\).

Throughout this section, we fix a Borel set \(E\subset X\) and let \(B=B(x_0, R)\) be an open ball of radius \(0< R < 2 \diam X\). We claim that
\begin{equation}\label{eq:main-goal-section-6}
\min\bigl\{\Haus^n\big(E\cap B\big), \Haus^n\big(B \setminus E\big) \bigr\} \leq C \cdot\mathscr{M}_-\big( E\, | \, 6 B\big)^{\frac{n}{n-1}}
\end{equation}
for some constant \(C\) depending only on the data of \(X\). We set \(D=\max\{ D_{n-1}, D_{n-2}\}\), where  \(D_k\) denotes the constant from Theorem~\ref{thm:filling-inj-spaces}, and we fix \(\epsilon_0>0\) sufficiently small such that
\begin{equation}\label{ef:def-of-epsilon-0}
(16 D^2 C_2)\cdot \epsilon_0^{\frac{1}{n-1}} \leq \min\Big\{ \lambda, \frac{1}{100\cdot C_1 }\Big\}.
\end{equation}
We may  assume that \(\Haus^n(E\cap B)\), \(\Haus^n(B\setminus E)\neq 0\) as well as 
\begin{equation}\label{eq:assumption-on-P}
P\coloneqq \mathscr{M}_-(E\, |\ 6B)\leq  \epsilon_0 R^{n-1}.
\end{equation}
Hence, it follows that $\Haus^n\big((\partial E \setminus E)\cap 6B\big) = 0$. Indeed,  $E_r\setminus E$ contains $\partial E\setminus E$ and if $\Haus^n\big((\partial E\setminus E\big)\cap 6B)>0$ then $\mathscr{M}_-(E\,|\, 6B) = \infty$. Thus, after replacing \(E\) by its closure, we may assume that \(E\) is compact. In particular, we have \(\Haus^n(B\setminus E_r)\to \Haus^n(B\setminus E) \) as \(r\searrow 0\). 

For almost every \(s\in \R\), we know that \(\partial(T\on E_s)\) is an integral \((n-1)\)-current in \(X\) whose support is contained in \(\{ x\in X : d(x, E)=s\}\), and
\[
\int_0^r \norm{\partial(T\on E_s)}(6B)\, ds \leq \norm{T}\big((E_r \setminus E)\cap 6B\big)
\]
for all \(r>0\). Let \(0 < \delta <\min\big\{C_2\epsilon_0 R^{n-1}, 1\big\}\). The parameter \(\delta\) is introduced mainly to account for the possibility that \(P=0\). By the above, there exists \(r>0\) arbitrarily small such that
\(T\on E_r\in \bI_n(X)\) with \(\spt \partial(T\on E_r)\subset \{ x\in X : d(x, E)=r\}\), and
\begin{equation}\label{eq:mass-partial-V}
\norm{\partial(T\on E_r)}(6B) \leq C_2 P+\delta,
\end{equation}
as well as \(\Haus^n(B\setminus E)\leq 2\Haus^n(B\setminus E_{3r})\) and \(\Haus^{n-1}\big(\{ x\in X : d(x, E)=r\}\big)<\infty\). We abbreviate 
\[
V=\partial(T\on E_r)
\]
and let \(R'\in (5R, 6R)\) be such that \(V\on B(x_0, R')\in \bI_{n-1}(X)\) and the support of the current \(\partial(V\on B(x_0, R'))\) is contained in \(S(x_0, R')=\{ x\in X : d(x_0, x)=R' \}\).

By Theorem~\ref{thm:filling-inj-spaces}, there exists a minimal filling \(U\in \bI_{n-1}(E(X))\) of \(\partial(V\on B(x_0, R'))\) and any such \(U\) satisfies 
\[
\spt U \subset N_{D\mass(U)^{\frac{1}{n-1}}}(\spt \partial U).
\]
Since \(U\) is a minimal filling, 
\[
\mass(U) \leq \mass(V \on B(x_0, R')) \leq \norm{V}(6B)\leq C_2P+\delta,
\]
where we used \eqref{eq:mass-partial-V} for the last inequality.
Thus, because of \eqref{eq:assumption-on-P} and the definition of \(\delta\), we obtain 
\begin{equation}\label{eq:support-of-U-is-close}
\spt U \subset N_{\epsilon_1 R}(\spt \partial U),
\end{equation}
where \(\epsilon_1=D\cdot(2 C_2 \epsilon_0)^{\frac{1}{n-1}}\). Since \(\spt \partial U\subset S(x_0, R')\) and \(\epsilon_1<1\), it follows that 
\begin{equation}\label{eq:spt-of-U-does-not-meet-4R-ball}
\spt U \cap B^{E(X)}(x_0, 4R)=\varnothing.
\end{equation}
Now, we consider the integral \((n-1)\)-cycle \(W=V\on B(x_0, R')-U\) in \(E(X)\). It follows that \(\mass(W)\leq 2(C_2 P+\delta)\) as well as \(\spt W \subset N_{\epsilon_1 R}(X)\). Using Theorem~\ref{thm:filling-inj-spaces} once again, we find a minimal filling \(S\in \bI_n(E(X))\) of \(W\) satisfying \(\mass(S)\leq D \mass(W)^{\frac{n}{n-1}}\) and \(\spt S \subset N_{D \mass(S)^{\frac{1}{n}}}(\spt W)\). It follows that 
\begin{equation}\label{eq:spt-of-S-is-close}
\spt S \subset N_{\epsilon_2 R}(\spt W),
\end{equation}
where \(\epsilon_2=D^{\frac{n+1}{n}}\cdot (4 C_2 \epsilon_0)^{\frac{1}{n-1}}\), and thus \(\spt S \subset N_{(\epsilon_1+\epsilon_2)R}(X)\). Moreover, we know that if \(y\in \spt S\) and \(0 \leq s \leq d(y, \spt W)\), then
\begin{equation}\label{eq:lower-bound-for-mass-of-S}
\norm{S}(B(y, s))\geq D^{-(n-1)} s^n.
\end{equation}
Let 
\[
A=(\spt U \cap X) \cup \big\{x\in X : d(x, E)=r\big\}.
\]
Then \(A\subset X\) is non-empty and compact, \(\dim_N(A)\leq N\), and \(\spt W \subset A \cup \spt U\). Notice that \(\Haus^{n-1}(A)<\infty\). Indeed,  by our assumption on \(r\), the Hausdorff $(n-1)$-measure of $\{ x\in X : d(x, E)=r\}$ is finite, and since \(U\) is a minimal filling, it follows that \(\spt U \setminus \spt \partial U \subset \set U\) and thus 
\[
\Haus^{n-1}(\spt U)\leq \Haus^{n-1}(\spt \partial U)+\Haus^{n-1}(\set U) < \infty,
\]
where we used that \(\spt \partial U \subset \spt V \subset \{ x\in X : d(x, E)=r\}\). Set \(a = \frac{1}{4}\) and \(b=20\). By applying Lemma~\ref{lem:Nagata-cover} to \(\spt S \cup A\), we find that there exists \(F\subset \spt S \setminus A\) such that
\[
\spt S\setminus A \subset \bigcup_{y\in F} B^{E(X)}(y, b\cdot r_y),
\]
where \(r_y=d(y, A)\), and \(\{ B^{E(X)}(y, a\cdot r_y)\}_{y\in F}\) has multiplicity at most \(L\) for some \(L\geq 1\) depending only on the data of $X$. 
Since \(A\) is compact and \(\Haus^{n}(A)=0\), there exist open balls \(B(z_i, s_i)\), \(i=1, \ldots, i_0\), with \(z_i\in A\), such that their union covers \(A\) and 
\begin{equation}\label{eq:sum-of-s-i-is-small}
\sum_{i=1}^{i_0} s_i^n \leq \delta.
\end{equation}
Hence, as \(E(X)\) is compact and thus \(\spt S\) is compact as well, there exist \(y_1, \ldots, y_{j_0}\in F\) such that
\begin{equation}\label{eq:containment-of-spt-of-S}
\spt S \subset \bigcup_{i=1}^{i_0} B^{E(X)}(z_i, s_i)\cup \bigcup_{j=1}^{j_0} B^{E(X)}(y_j, b\cdot r_j),
\end{equation}
where \(r_j=d(y_j, A)\). For later use, we record that the \(r_j\)'s are uniformly bounded by \(R\).

\begin{lemma}\label{lem:unifrom-bound-of-the-rjs}
 For every \(j=1, \ldots, j_0\), we have \(r_j \leq (\epsilon_1+\epsilon_2) R\). 
\end{lemma}

\begin{proof}
 Fix \(j\) and let \(x\in \spt W\) be such that \(d(y_j, x)=d(y_j, \spt W)\). By \eqref{eq:spt-of-S-is-close}, it follows that \(d(y_j, x)\leq \epsilon_2 R\). If \(x\in A\), then \(d(y_j, A) \leq d(y_j, x)\)
and so \(r_j \leq \epsilon_2 R\), as desired. Therefore, we are left to consider the case when \(x\notin A\).  If \(x\notin A\), then \(x\in \spt U \setminus X\) and due to \eqref{eq:support-of-U-is-close} there exists \(x'\in \spt \partial U \) such that \(d(x', x) \leq \epsilon_1 R\). Notice that \(\spt \partial U \subset\spt V \subset \{ x\in X : d(x, E)=r\} \subset A\) and thus \(x'\in A\). Hence, 
\[
d(y_j, x')\leq d(y_j, x)+d(x, x')\leq (\epsilon_1+\epsilon_2) R
\]
and so \(r_j=d(y_j, A)\leq d(y_j, x') \leq (\epsilon_1+\epsilon_2) R\), as was to be shown.
\end{proof}
We abbreviate
\[
h=\frac{1}{3}\min\big\{R, r, s_1, \ldots, s_{i_0}, r_1, \ldots, r_{j_0}\big\} 
\]
and let \(\epsilon>0\) be as in Lemma~\ref{lem:technical-inclusions-approx-homeo} applied to our fixed homeomorphism \(\rho\colon X \to M\). There exists $\varepsilon'>0$ such that 
$
N_{\varepsilon'}(\varrho(A'))\subset \varrho(N_h(A'))
$ 
for every $A'\subset X$. Let $$\varphi\colon \bar{N}_{\lambda\diam X}^{E(X)}(X)\to M$$ be a Lipschitz map satisfying 
$
d(\varphi, \varrho\circ\pi)<\min\{\varepsilon, \varepsilon'\}
$
and $\varphi_\#T = \bb{M}$. We consider the integral $n$-current in $M$ given by $Q:= \varphi_\#(T\on E_r - S)$ and claim that \(\spt(\partial Q)\subset \varrho(X\setminus B(x_0, 2R))\). In order to prove this, notice first that 
$$
\pi\big(N_{\varepsilon_1R}^{E(X)}(X)\setminus B(x_0, 4R)\big)\subset X\setminus B(x_0, 3R),
$$ 
as is easy to prove. Since $\partial(T\on E_r-S)=V-W = V\on(X\setminus B(x_0, R')) + U$
it moreover follows that 
$$
\spt(\partial(T\on E_r - S))\subset N_{\varepsilon_1R}^{E(X)}(X)\setminus B(x_0, 4R).
$$ 
Consequently, we obtain 
\begin{align*}
 \spt(\partial Q)&\subset N_{\varepsilon'}\bigl(\varrho(\pi(\spt \partial(T\on E_r- S)))\bigr)\\
 &\subset \varrho\big(N_h(X\setminus B(x_0, 3R))\big)\subset \varrho\big(X\setminus B(x_0, 2R)\big),
\end{align*}
proving the claim. It therefore follows from the constancy theorem, see \cite[Corollary~3.13]{federer-1960}, that 
\[
Q\on \rho(B)=c\, \bb{M}\on \rho(B)
\]
for some \(c\in \Z\). The following two possible inclusions result from considering the cases \(c=0\) and \(c\neq 0\) separately.

\begin{lemma}\label{lem:cases-constancy-theorem}
We have $B \cap E_{r-h}\subset N_h(\pi(\spt S))$ or $B\setminus E_{3r}\subset N_h(\pi(\spt S))$.
\end{lemma}

\begin{proof}
To begin notice that $$\varphi(\spt S)\subset N_{\varepsilon'}(\varrho(\pi(\spt S)))\subset \varrho(N_h(\pi(\spt S)))$$ and
\[
\big(\varphi_{\#}S\big)\on \rho(B)=\big(\varphi_{\#}(T\on E_r)-c\bb{M}\big)\on \rho(B).
\]
Suppose now that \(c\neq 0\). Since \(\varphi(\overline{E_r})\subset \varphi(E_{2r})\subset \rho(E_{2r+h})\subset \rho(E_{3r})\), it follows that \(\rho(X\setminus E_{3r})\subset \spt\big(\varphi_{\#}(T\on E_r)-c\bb{M}\big) \) and therefore
$$\rho(B \setminus E_{3r})\subset\spt\big(\varphi_{\#}(T\on E_r)-c\bb{M}\big)\cap \rho(B)\subset \varphi(\spt S)\subset \varrho(N_h(\pi(\spt S))),$$
which shows that \(B\setminus E_{3r}\subset N_h(\pi(\spt S))\) if \(c\neq 0\). Second, we suppose that \(c=0\). Since \(\varphi(X \setminus E_r)\subset  \rho(X \setminus E_{r-h})\) and  \(\varphi_{\#}(T\on E_r)=\bb{M}-\varphi_{\#}(T\on (X\setminus E_r))\), we find that \(\rho(E_{r-h})\subset \spt \varphi_{\#}(T\on E_r)\). Hence, because \(\rho(B \cap E_{r-h})=\rho(E_{r-h}) \cap \rho(B)\), the same argument as above shows that \(\rho(B \cap E_{r-h})\subset \varphi(\spt S)\subset \varrho(N_h(\pi(\spt S)))\), which shows that $B \cap E_{r-h}\subset N_h(\pi(\spt S))$ in case $c=0$.
\end{proof}

Let \(x_j\in A\) be such that \(d(x_j, y_j)\leq 2r_j\) for all \(j=1, \ldots, j_0\). Since \(\pi(B^{E(X)}(y_j, b r_j))\subset B(x_j, 3 C_1 b r_j)\) and \(\pi(B^{E(X)}(z_i, s_i))\subset B(z_i, C_1 s_i)\), it follows from \eqref{eq:containment-of-spt-of-S} and the choice of $h$ that 
\begin{equation}\label{eq:pi-spt-S-in-union-of-balls}
    N_h(\pi(\spt S))\subset \bigcup_{i=1}^{i_0} B(z_i, 2C_1 s_i) \cup \bigcup_{j=1}^{j_0} B(x_j, 4 C_1 b r_j).
\end{equation}
Suppose \(j\) is such that \(d(y_j, x_0)>3R\). Using Lemma~\ref{lem:unifrom-bound-of-the-rjs}, we get \(d(x_j, x_0)> 3R-2 r_j\geq (5/2) R\), since \(2(\epsilon_1+\epsilon_2) \leq 1/2\). In particular,
\[
 B(x_j, 4 C_1 b r_j) \cap B(x_0, 2R)=\varnothing,
\]
as \(4 C_1 b\cdot (\epsilon_1+\epsilon_2) \leq 1/2\). It follows with the above that $N_h(\pi(\spt S))\cap 2B\subset A_*$, where
$$
A_*= \bigcup_{i=1}^{i_0} B(z_i, 2C_1 s_i) \,\cup \bigcup_{j \, :\, d(y_j, x_0)\leq 3R} B(x_j, 4 C_1 b r_j).
$$
In the following lemma we give an upper bound on the Hausdorff \(n\)-measure of \(A_*\). 

\begin{lemma}\label{lem:upperbound-of-A-star}
 There exists a constant \(C_3>0\) depending only on the data of \(X\) such that
\[
\Haus^n(A_*)\leq C_3\cdot P^{\frac{n}{n-1}}+C_3\delta.
\]
\end{lemma}

\begin{proof}
We claim that 
\begin{equation}\label{eq:assumption-on-r-j}
r_j \leq d(y_j, \spt \partial S)
\end{equation}
for all \(j\) with \(d(y_j, x_0)\leq 3R\). Suppose for the moment that \eqref{eq:assumption-on-r-j} holds true. Then, by the use of \eqref{eq:lower-bound-for-mass-of-S}, we get
\[
\sum_{j : d(y_j, x_0)\leq 3R} r_j^n \leq (D/a)^n\sum_{j : d(x_j, x_0)\leq 3R} \norm{S}(B(y_j, ar_j)) \leq L (D/a)^n \mass(S)
\]
and thus using that for every ball \(B(x, r)\), we have \(\Haus^n(B(x, r))\leq \alpha r^n\) for some uniform constant \(\alpha\), it follows that
\[
\Haus^n(A_*)\leq \alpha (2 C_1)^n \sum_{i=1}^{i_0} s_i^n+\alpha L  (4b C_1 D/a)^n \mass(S).
\]
By construction, \(\mass(S)\leq D \mass(W)^{\frac{n}{n-1}} \leq D (4C_2)^{\frac{n}{n-1}}\big(P^{\frac{n}{n-1}}+\delta \big)\), where we used the convexity of the function \(x\mapsto x^{\frac{n}{n-1}}\) and that \(\delta<1\). Hence, because of \eqref{eq:sum-of-s-i-is-small}, we arrive at
\[
\Haus^n(A_*)\leq C_3\big(P^{\frac{n}{n-1}}+\delta \big),
\]
as desired. To finish the proof it thus remains to show \eqref{eq:assumption-on-r-j}. Fix \(j\) with \(d(y_j, x_0)\leq 3R\). Due to \eqref{eq:spt-of-U-does-not-meet-4R-ball}, it follows that \(R \leq d(y_j, \spt U)\). Recall that \(\partial S=W\). Since \(\spt W \subset A\cup \spt U \), 
\begin{equation}\label{eq:auxiliary-22}
\min\{ d(y_j, A), d(y_j, \spt U)\} \leq d(y_j, \spt W).
\end{equation}
By Lemma~\ref{lem:unifrom-bound-of-the-rjs}, it follows that \(d(y_j, A) \leq (\epsilon_1+\epsilon_2) R \leq R \leq d(y_j, \spt U)\). Therefore, due to \eqref{eq:auxiliary-22}, we find that \(r_j \leq d(y_j, \spt W)\), as claimed.
\end{proof}

Now, we are in a position to finish the proof of Theorem~\ref{thm:relative-isoperimetric-inequality}. By Lemma~\ref{lem:cases-constancy-theorem} and \eqref{eq:pi-spt-S-in-union-of-balls}, it follows that
\[
B \cap E_{r-h} \subset A_* \quad \text{ or } \quad B\setminus E_{3r} \subset A_*.
\]
Hence, by the use of Lemma~\ref{lem:upperbound-of-A-star},
\[
\min\bigl\{\Haus^n(B \cap E_{r-h}), \Haus^n(B \setminus E_{3r}) \bigr\} \leq C_3 \cdot P^{\frac{n}{n-1}}+C_3\delta,
\]
and so by our assumptions on \(r\) and since \(\delta\) can be made arbitrarily small this yields \eqref{eq:main-goal-section-6} and completes the proof of Theorem~\ref{thm:relative-isoperimetric-inequality}.

We finally provide the proof of Semmes' theorem about the validity of a Poincar\'e inequality.

\begin{proof}[Proof of Corollary~\ref{cor:Semmes-theorem-Poincare}]
 Let $X$ be as in the statement of the corollary. It follows from Theorem~\ref{thm:relative-isoperimetric-inequality} and the Ahlfors $n$-regularity of $X$ that there exist constants $C,\lambda\geq 1$ such that $X$ satisfies the strong relative isoperimetric inequality $$\min\bigl\{\Haus^n\big(E\cap B\big), \Haus^n\big(B \setminus E\big) \bigr\} \leq C R\cdot  \mathscr{M}_-\big( E\, | \, \lambda B\big)$$ for every Borel set $E\subset X$ and every ball $B(x,R)$ of radius $R>0$. Now, it is well-known that doubling metric spaces with a strong relative isoperimetric inequality support a weak $1$-Poincar\'e inequality; see \cite[Theorem 1.1]{Bobkov-Houdre-1997} or for example \cite{Kinnunen-Korte-Shanmugalingam-Tuominen-2012, Korte-Lahti-2014}. Therefore, $X$ supports a weak $1$-Poincar\'e inequality, which concludes the proof.
\end{proof}


\section{Lipschitz-volume rigidity for metric manifolds}\label{sec:section8}
In this final section we combine Z\"ust's results  \cite{Zuest-23} with our main existence results (Theorems~\ref{thm:main} and \ref{thm:existence-intcurr-dimension2-intro}) in order to establish Theorem~\ref{thm:Lipschitz-Volume-rigidity}. Before proving Theorem~\ref{thm:Lipschitz-Volume-rigidity} in its generality, we first consider the special case of quasiconvex metric spaces homeomorphic to a $2$-dimensional manifold. The proof for such spaces is more accessible, since we do not need to use Theorem~\ref{thm:existence-intcurr-dimension2-intro}, but can establish the existence of a non-trivial integral current (needed to apply Z\"ust's result) with more elementary methods. The auxiliary result, Lemma~\ref{lem:good-integral-cycle-LipVol} below, which is used for the proof of the special case, will also be used in the proof of the general case. Recall that a metric space $X$ is \textit{quasiconvex} if there exists $L\geq 1$ such that any two points in $X$ can be joined by a curve of length at most $L$ times their distance.

\begin{theorem}\label{thm:Lipschitz-Volume-rigidity-dim2-quasiconvex}
Let $M$ be a closed, orientable, Riemannian surface and let $X$ be a quasiconvex metric space homeomorphic to a closed orientable surface. If $\Haus^2(X)=\Haus^2(M)$, then every surjective $1$-Lipschitz map from $X$ to $M$ is an isometric homeomorphism.
\end{theorem}

We begin with a preparatory lemma for which we need a different notion of mass for integer rectifiable currents. Let $T$ be an integer rectifiable $n$-current in $X$. There exist compact sets $K_i\subset \R^n$, functions $\theta_i\in L^1(\R^n, \Z)$ with $\spt \theta_i\subset K_i$, and bi-Lipschitz maps $\rho_i\colon K_i\to X$ such that 
\begin{equation}\label{eq:rep-int-rect-current}
T=\sum_{i\in \N} \rho_{i\#}\bb{\theta_i} \quad \text{ and } \quad \mass(T)=\sum_{i\in \N} \mass(\rho_{i\#}\bb{\theta_i}).
\end{equation}
We may also assume that the images $\rho_i(K_i)$ are pairwise disjoint.
We call any such collection \((K_i, \theta_i, \rho_i)\) a  \textit{parametrization} of \(T\). 
The Hausdorff or Busemann mass of $T$ is defined by 
$$
\mass^{\rm b}(T)\coloneqq \sum_{i\in\N}\int_{K_i}|\theta_i| \,\mathbf{J}(\md\rho_i)\,d\mathscr{L}^n.
$$
By applying the area formula, it is easy to check that this definition is independent of the particular parametrization.

\begin{lemma}\label{lem:good-integral-cycle-LipVol}
 Let $X$ be a complete metric space, \(S\in\bI_n(X)\) a non-trivial cycle, and $M$ a closed, oriented Riemannian $n$-manifold such that $\Haus^n(M)=\Haus^n(X)$. If $\varphi\colon X\to M$ is a surjective $1$-Lipschitz map, then there exists a cycle $T\in\bI_n(X)$ with \(\spt T=\spt S\) and such that $\varphi_\#T = \bb{M}$ and $\mass^{\rm b}(T)\leq \Haus^n(X)$.
\end{lemma}

\begin{proof}
Let \((K_i, \theta_i, \rho_i)\) be a parametrization of \(S\).
For $(f,\pi)\in\mathcal{D}^n(M)$ we calculate using the area formula that
\begin{equation*}
 \begin{split}
   \varphi_\#S(f,\pi) &= \sum_{i\in\N}\int_{K_i} \theta_i(x) \cdot (f\circ\varphi\circ\rho_i)(x)\cdot \det(D(\pi\circ\varphi\circ\rho_i)(x))\,d\mathscr{L}^n(x)\\
   &= \int_M \theta(y)\, f(y)\,\det(D\pi(y))\,d\hspace{-0.14em}\Haus^n(y),
 \end{split}
\end{equation*}
where we have set 
$$
\theta(y)\coloneqq \sum_{i\in\N} \sum_{x\in K_i\,:\, \varphi(\rho_i(x))=y} \theta_i(x)\,\sgn(\det(D(\varphi\circ\rho_i)(x))).
$$
Since $\varphi_\#S$ is an integral $n$-cycle in $M$ there exists $k\in\Z$ such that $\varphi_\#S = k\cdot\bb{M}$ and thus $|\theta|=|k|$ almost everywhere on $M$.

Finally, since $\Haus^n(X)=\Haus^n(M)$ and $\varphi$ is a surjective $1$-Lipschitz map, it follows that $\varphi$ preserves the Hausdorff measure of all measurable sets and, by the coarea inequality, that for almost every $y\in M$ the preimage $\varphi^{-1}(y)$ consists of exactly one point. This implies that for every $i\in \N$ we have $|\theta_i|=|k|$ almost everywhere on $K_i$. In particular, $k\not=0$ and the current $T:= (1/k)\cdot S$ is still integer rectifiable with $\partial T=0$, hence $T\in\bI_n(X)$, and we have \(\spt T=\spt S\) and $\mass^{\rm b}(T)\leq \Haus^n(X)$, as well as $\varphi_\#T = \bb{M}$  by construction.
\end{proof}

We can now prove Theorem~\ref{thm:Lipschitz-Volume-rigidity-dim2-quasiconvex} without relying on Theorem~\ref{thm:existence-intcurr-dimension2-intro}.

\begin{proof}[Proof of Theorem~\ref{thm:Lipschitz-Volume-rigidity-dim2-quasiconvex}]
Since $X$ is quasiconvex it follows from \cite{Jorgensen-Lang-22} that \(\dim_N(X)=2\). By \cite[Theorem 1.6]{basso2021undistorted} there exists $C\geq 1$ and for every $k\in \N$ a finite metric simplicial complex $\Sigma_k$ and $C$-Lipschitz maps $\psi_k\colon X\to \Sigma_k$ and $\varphi_k\colon\Sigma_k\to E(X)$ such that $\varphi_k(\Sigma_k)$ lies in the $(1/k)$-neighborhood of $X$ in $E(X)$ and $d(x, \varphi_k(\psi_k(x)))\leq 1/k$ for all $x\in X$.

By hypothesis, $X$ is homeomorphic to a closed orientable surface $M'$. We equip $M'$ with an orientation and with a Riemannian metric. Let $\rho\colon M'\to X$ be a positively oriented homeomorphism. One shows exactly as in the proof of \cite[Proposition 5.1]{Meier-Wenger} that there exist $L\geq 0$ and for each $k\in\N$ a Lipschitz map $\rho_k\colon M'\to \Sigma_k$ with $$\vol(\rho_k)\leq L$$ and such that $d(\rho_k, \psi_k\circ\rho)$ is as small as we want. Define a cycle $S_k\in \bI_2(E(X))$ by $S_k= (\varphi_k\circ\rho_k)_\#\bb{M'}$ and notice that $\spt(S_k)\subset \bar{N}_{1/k}(X)$ as well as 
$$
\mass(S_k)\leq C^2 \mass(\rho_{k\,\#}\bb{M'})\leq C'L
$$ 
for some $C'$ only depending on $C$. Since $E(X)$ is compact we may assume, after possibly passing to a subsequence, that $S_k$ converges weakly to some cycle $S\in\bI_2(X)$. 

We claim that for every Lipschitz map $\varphi\colon X\to M'$, we have
$$
\varphi_\#S = \deg(\varphi)\bb{M'}.
$$ 
In order to see this, we first extend $\varphi$ to a Lipschitz map $\bar{\varphi}\colon N_\delta(X)\to M'$ defined on some neighborhood $N_\delta(X)$ in $E(X)$. Notice that if $k$ is large enough and $d(\rho_k, \psi_k\circ\rho)$ is sufficiently small then $\bar{\varphi}\circ\varphi_k\circ\rho_k$ is homotopic to $\varphi\circ\rho$ and hence 
$$
\deg(\bar{\varphi}\circ\varphi_k\circ\rho_k) = \deg(\varphi).
$$ 
It then follows that 
$$
\bar{\varphi}_\#S_k = \deg(\bar{\varphi}\circ\varphi_k\circ\rho_k)\cdot\bb{M'} = \deg(\varphi)\cdot\bb{M'}
$$ 
and since $\bar{\varphi}_\#T_k$ converges weakly to $\varphi_\#T$, we conclude that $\varphi_\#T = \deg(\varphi)\cdot \bb{M}$ as claimed. It now follows from Proposition~\ref{prop:support-is-all-of-X} that $\spt S = X$.

Finally, let $M$ be a closed, orientable, Riemannian surface with $\Haus^2(X) = \Haus^2(M)$, and suppose $\varphi\colon X\to M$ is a surjective $1$-Lipschitz map. We equip $M$ with an orientation. By Lemma~\ref{lem:good-integral-cycle-LipVol} there exists a cycle $T\in\bI_2(X)$ with $\spt T = \spt S = X$ and such that $\varphi_\#T=\bb{M}$ and $\mass^{\rm b}(T)\leq \Haus^2(X)$. It now follows from Z\"ust's rigidity result \cite[Theorem 1.2]{Zuest-23} that $\varphi$ is an isometric homeomorphism.
\end{proof}

We mention here that it is not difficult to prove that the cycle $S$ in the proof above moreover satisfies $\norm{S}\leq C\Haus^2$ for some universal constant and is a generator of the homology group via integral currents. This thus provides a simple approach to Theorem~\ref{thm:existence-intcurr-dimension2-intro} in the quasiconvex case which does not rely on the uniformization result \cite{Ntalampekos-Romney-22}. We finally provide the proof of our general Lipschitz-volume rigidity result stated in the introduction.

\begin{proof}[Proof of Theorem~\ref{thm:Lipschitz-Volume-rigidity}]
 By Theorems~\ref{thm:main} and \ref{thm:existence-intcurr-dimension2-intro} and Proposition~\ref{prop:support-is-all-of-X} there exists an integral $n$-cycle in $X$ whose support is all of $X$. Thus, Lemma~\ref{lem:good-integral-cycle-LipVol} implies that there exists a cycle $T\in\bI_n(X)$ with $\spt T=X$ and such that $\varphi_\#T=\bb{M}$ and $\mass^{\rm b}(T)\leq \Haus^n(X)$. It now follows from Z\"ust's \cite[Theorem 1.2]{Zuest-23} that $\varphi$ is an isometric homeomorphism.
\end{proof}


\let\oldbibliography\thebibliography
\renewcommand{\thebibliography}[1]{\oldbibliography{#1}
\setlength{\itemsep}{3pt}}

\bibliographystyle{plain}

\end{document}